\documentclass[10pt, a4paper]{amsart}

\usepackage{latexsym,amsmath,amsfonts,amscd,amssymb, mathrsfs, amsthm}
\usepackage[english]{babel}
\usepackage{appendix}
\usepackage{hyperref}

\advance\oddsidemargin by -0.5cm
\advance\evensidemargin by -0.5cm
\advance\topmargin by -1.0cm 

\theoremstyle{plain}  
\newtheorem{theorem}{Theorem}[section]

\newtheorem*{theorem*}{Theorem}

\newtheorem{corollary}[theorem]{Corollary}
\newtheorem{lemma}[theorem]{Lemma}
\newtheorem{proposition}[theorem]{Proposition}

\theoremstyle{definition}
\newtheorem{definition}[theorem]{Definition}
\newtheorem{example}[theorem]{Example}
\newtheorem{remark}[theorem]{Remark}

\newtheorem*{claim*}{Claim}

\numberwithin{equation}{section}

\newcommand{\R}{\mathbb{R}}

\newcommand{\N}{\mathbb{N}}

\newcommand{\g}{\mathfrak{g}}
\newcommand{\z}{\mathfrak{z}}
\newcommand{\p}{\mathfrak{p}}

\newcommand{\mk}{\mathfrak{k}}
\newcommand{\n}{\mathfrak{n}}

\newcommand{\bdot}{\boldsymbol{\cdot}}
\newcommand{\e}{\mathrm{e}}

\newcommand{\Id}{\mathrm{I}}
\newcommand{\ad}{\mathrm{ad}}
\newcommand{\Ad}{\mathrm{Ad}}
\newcommand{\GL}{\mathrm{GL}}
\newcommand{\Sym}{\mathrm{Sym}}
\newcommand{\Aut}{\mathrm{Aut}}

\usepackage{xcolor}
\definecolor{MyBlue}{RGB}{0,0,255}

\definecolor{MyRed}{RGB}{255,0,0}

\definecolor{MyGray}{RGB}{150,60,60}

\newcommand{\be}{\begin{equation}}
\newcommand{\ee}{\end{equation}}

\newcommand{\ben}{\begin{enumerate}}
\newcommand{\een}{\end{enumerate}}
\newcommand{\bit}{\begin{itemize}}
\newcommand{\eit}{\end{itemize}}
\newcommand{\edoc}{\end{document}}

\def\br#1\er{{#1}} 
\def\bw#1\ew{\textcolor{brown}{#1}} 
\def\bb#1\eb{\textcolor{blue}{#1}} 
\def\br#1\er{\textcolor{red}{#1}} 
\def\bm#1\em{\textcolor{magenta}{#1}}
\def\bv#1\ev{\textcolor{olive}{#1}}

\makeatletter
\renewcommand{\tocsection}[3]{%
	\indentlabel{\@ifnotempty{#2}{\ignorespaces#1 #2\quad}}#3}
\renewcommand{\tocsubsection}[3]{%
	\indentlabel{\@ifnotempty{#2}{\ignorespaces#1 #2\quad}}#3}

\newcommand\@dotsep{4.5}
\def\@tocline#1#2#3#4#5#6#7{\relax
	\ifnum #1>\c@tocdepth 
	\else
	\par \addpenalty\@secpenalty\addvspace{#2}%
	\begingroup \hyphenpenalty\@M
	\@ifempty{#4}{%
		\@tempdima\csname r@tocindent\number#1\endcsname\relax
	}{%
		\@tempdima#4\relax
	}%
	\parindent\z@ \leftskip#3\relax \advance\leftskip\@tempdima\relax
	\rightskip\@pnumwidth plus1em \parfillskip-\@pnumwidth
	#5\leavevmode\hskip-\@tempdima{#6}\nobreak
	\leaders\hbox{$\m@th\mkern \@dotsep mu\hbox{.}\mkern \@dotsep mu$}\hfill
	\nobreak
	\hbox to\@pnumwidth{\@tocpagenum{\ifnum#1=1\fi#7}}\par
	\nobreak
	\endgroup
	\fi}
\AtBeginDocument{%
\expandafter\renewcommand\csname r@tocindent0\endcsname{0pt}
}
\def\l@subsection{\@tocline{2}{0pt}{2.5pc}{5pc}{}}
\makeatother

\begin{document}
\title[Lie groups with all metrics complete]{Lie groups with all left-invariant  semi-Riemannian metrics complete }

\author[A. Elshafei]{Ahmed Elshafei}
\address{\hspace{-5mm} Ahmed Elshafei, Centro de Matem\'{a}tica,
Universidade do Minho,
Campus de Gualtar,
4710-057 Braga,
Portugal} 
\email {a.el-shafei@hotmail.fr}

\author[A.C. Ferreira]{Ana Cristina Ferreira}
\address{\hspace{-5mm} Ana Cristina Ferreira, Centro de Matem\'{a}tica,
Universidade do Minho,
Campus de Gualtar,
4710-057 Braga,
Portugal} 
\email {anaferreira@math.uminho.pt}

\author[M. S\'anchez]{Miguel S\'anchez}
\address{\hspace{-5mm} Miguel S\'anchez, Departamento de Geometr\'ia y Topolog\'ia, Facultad de Ciencias, Universidad de Granada, Campus Fuentenueva s/n, 18071 Granada, Spain}
\email{sanchezm@ugr.es}

\author[A. Zeghib]{Abdelghani Zeghib}
\address{\hspace{-5mm} Abdelghani Zeghib, UMPA, CNRS, \'Ecole Normale Sup\'erieure de Lyon, 46, All\'ee d'Italie 69364 Lyon Cedex 07, France }
\email{abdelghani.zeghib@ens-lyon.fr}
\subjclass[2020]{Primary 53C22; Secondary 53C30, 53C50}

\date{\today}

\begin{abstract}
For  each  left-invariant  semi-Riemannian metric $g$ on a Lie group $G$, we introduce  the class of bi-Lipschitz Riemannian   {\em Clairaut}  metrics, whose completeness implies the  completeness of $g$. When the  adjoint representation of $G$ satisfies an at most linear  growth bound,  then all the Clairaut metrics are complete for any $g$. 
We prove that this bound is satisfied by   compact and 2-step nilpotent groups, as well as by semidirect products $K \ltimes_\rho \R^n$ , where $K$ is the direct product of a compact and an abelian Lie group and $\rho(K)$ is pre-compact; they include  all the known examples of  Lie groups with all left-invariant metrics complete.   
 The  affine group of the real line  is considered to illustrate   how our techniques work even in  
the absence of linear  growth  and suggest new questions. 
 
\end{abstract}

\vspace*{-3mm}

\maketitle

\tableofcontents

\section{Introduction}
 
 A striking difference between Riemannian and indefinite semi-Riemannian left-invariant metrics on a Lie group $G$ is the possibility of (geodesic) incompleteness for the latter.  Abundant examples are known  \cite{Guediri, BM}; indeed, even the affine group of the real line exhibits this property (see Section \ref{s6} below).  Here, we study conditions which ensure the completeness of every  left-invariant metric $g$ on $G$. 
 
 With this aim, two notions which might have interest in their own right are introduced, namely, the {\em Clairaut metrics} associated with $g$ and the possible (at most)  {\em linear growth} applicable to $G$. The former is inspired by the techniques of a  celebrated theorem by Marsden [M73], which ensures the completeness of compact semi-Riemannian homogeneous manifolds, while  the latter comes from natural estimates of growth on the geodesic vector field implying completeness (see the review [S15, \S 3] or Section \ref{s3} below). Our goals can be neatly summarized as follows: 
\begin{enumerate}
\item\label{item1} Any left-invariant semi-Riemannian metric $g$ on the Lie group $G$ determines the class of its Clairaut metrics. These metrics are Riemannian and  bi-Lipschitz bounded, thus, determining  a single uniformity (Proposition \ref{prop: Clairaut uniform}). 
The completeness of this uniformity implies the completeness of $g$ (Theorem \ref{c02}).

\item \label{item2} For any Lie group $G$, the behaviour of the adjoint representation $\Ad: G\rightarrow \mathrm{GL}(\mathfrak{g})$  on the connected part of the identity defines  a natural notion of  (at most)  linear growth (Definition \ref{def:glg}). 
When $G$ is of linear growth,  the Clairaut metrics   
of every left-invariant metric $g$ on $G$ are complete (Theorem \ref{t_lg}) and, thus,  $g$ is complete from item \eqref{item1}. 

\item  The list of groups of at most linear growth includes compact ones,  2-step  nilpotent groups and  the semidirect products $K\ltimes_\rho \mathbb{R}^n$, where $K$ is the direct product of a compact and an abelian group, and  $\rho(K)$ is pre-compact in $\mathrm{GL}(n, \R)$. All the left-invariant semi-Riemannian metrics on these groups are complete from item \eqref{item2} and the above list includes  the known cases (contained in \cite{Marsden, Guediri, BM}).

\item A simple example as the affine group of the line exhibits the subtleties appearing when the linear growth assumption is dropped.

\end{enumerate}

Next let us explain briefly the ideas in these items.

 \subsection{Clairaut metrics} Marsden's proof 
 is based upon the use of Clairaut first integrals. 
 For $Y$ a Killing vector field on a semi-Riemannian manifold $(M, g)$, the 1-form $\omega_Y $, dual to $Y$, defines
 a function on $TM$, which is invariant under the geodesic flow of $g$. Under Marsden's hypotheses, this is enough to show that any tangent vector lies in a compact subset of the tangent bundle  $TM$ which is invariant under the geodesic flow; this implies the  completeness of any geodesic $\gamma$.  As pointed out in \cite{RS94}, completeness is also ensured if, when $\gamma$ is restricted to a finite interval, its velocity $\dot\gamma$ lies in a compact subset of $TM$. This permits to extend Marsden's theorem in several directions. 
 In particular,  the compactness of $M$ is dropped in some cases when  a complete Riemannian metric $h$ can be constructed from the semi-Riemannian one. In 
 \cite[Prop. 2.1]{RS94} such an $h$ is constructed by using a suitable set of conformal-Killing vector fields. 
  However, as it will be shown in Section \ref{s2}, the case of a left-invariant metric $g$ on $G$
  is especially well suited for going beyond these ideas.
 
Specifically, any basis $(e_i)$ of the Lie algebra $\mathfrak{g}$ yields a frame $(Y_i)$ of (right-invariant) Killing vector fields and their dual forms $(\omega^i)$ generate the Clairaut metric $h=\sum_i (\omega^i)^2$, as seen in Section~\ref{s2.2}. 
A change of basis would generate a Clairaut  metric $\tilde{h}$ bi-Lipschitz related with $h$ and, so, the same  Clairaut uniformity on $G$. In Section~\ref{s2.3}, we show that \begin{quote}
 {\em   completeness of the Clairaut uniformity implies completeness of $g$} (Theorem \ref{c02}); 
\end{quote}
 moreover, these ideas  also provide a variant to the proof of Marsden's theorem (Remark~\ref{r_Marsden}). 
  The following observations are in order. 

First (Section~\ref{s2.1}), if the left-invariant metric $g$ is chosen Riemannian (thus, trivially complete), then the uniformity associated with its distance agrees with the natural uniformity associated with the Lie group structure of $G$. However, in principle,   this uniformity is not equal to the Clairaut uniformity associated with $g$ (this suggests a further  study,  see the prospects subsection below). 

Second (Section \ref{s2.3}), the freedom in the choice of the Lie algebra basis $(e_i)$ for the construction of a Clairaut metric $h$, permits to use a convenient ``Wick rotated'' left-invariant Riemannian metric $\tilde g$ so that $(e_i)$ is orthonormal for $g, \tilde g$ and $h$ on $\mathfrak{g} (\equiv T_1G)$, which will yield technical simplifications (Proposition \ref{prop: Wick}). 

 Finally (Section~\ref{s2.4}),  
  notice that the Clairaut metrics $h$ are neither left nor right-invariant, in general.
 To better understand their  transformation law and the role played by the adjoint representation,  an abstract viewpoint is introduced;  subsequent consequences (Proposition \ref{p_aut_orbits}) will be used later.

  \subsection{Linear growth and completeness} In order to prove the completeness of the Clairaut metric $h$, a Wick rotated left-invariant Riemannian metric $\tilde g$ can be used. Then, the issue becomes a  purely Riemannian one, studied in Section \ref{s3}. Namely,  estimate the  possible  completeness of a Riemannian metric $h$ in terms of a  prescribed  complete one $g_R=\tilde g$.  As a consequence of a general result on so-called primary bounds, {\em  $h$ is complete if it is lowerly bounded by $\tilde g$ divided by an affine function of the distance to a point} 
   (Proposition \ref{p_03}).
  
    In Section \ref{s4}, this bound is implemented 
  for the Clairaut metric $h$ in terms of the growth of the adjoint map $\Ad_p$ with $p\in G$. Indeed, there is a natural notion of  (at most) linear growth for the operator norm  $\Vert \Ad_p\Vert$ (Definition \ref{def:glg}). Because of  the group structure, this type of growth implies a lower and an upper bound (Proposition \ref{p_linearg}). The linear growth can be suitably expressed in terms of eigenvalues of the self-adjoint operator $\Ad ^*\circ \Ad$ (see  \eqref{e:lambda+}), which is the key for a crucial technical characterization: {\em  the linear growth of $\Ad$ is equivalent to the required lower bound for the completeness of $h$} (Lemma \ref{l_essential}). 
    
    So, as the linear growth of $G$ is defined exclusively in terms of its Lie group structure, \begin{quote}
 {\em when $G$ is of linear growth, all its Clairaut metrics $h$ and, thus, all its left-invariant semi-Riemannian metrics $g$, are complete} (Theorem \ref{t_lg}).  
  \end{quote} 
  It is worth pointing out that the bounds associated with the linear growth preserve completeness but not the uniformities (Example \ref{Ex_Unpreserved uniformities}).
  What is more, when $G$ is not of linear growth then the growth of $\Ad$ is equal for all the Wick rotated left-invariant metrics, no matter their signature, in spite of the fact that  completeness of both the left-invariant $g$ and Clairaut $h$ do depend on signature (Remark \ref{r_BEWARE}, see  the prospects below).

\subsection{Extended list of groups with all left-invariant metrics complete}
A case by case study of the linear growth of groups yields our main result (Section \ref{s5}), namely:

\begin{theorem} \label{main}
A Lie group $G$ is  of linear growth if it lies  in one of the following classes
\begin{itemize}
\item[-] compact
\item[-] 2-step nilpotent (including  abelian) 
\item[-] the semidirect product $K\ltimes_\rho \mathbb{R}^n$ where $K$ is the direct product of a compact and an abelian group and such that $\rho(K)$ is pre-compact in $\mathrm{GL}(n, \R)$
\item[-] a subgroup or the direct product of the groups above,
 \end{itemize}
and, then,  every left-invariant semi-Riemannian metric on $G$ is complete.
\end{theorem}

 This contains as particular cases all known cases of groups having all their left-invariant semi-Riemannian metrics complete, as far as we know.
They are due to Marsden \cite{Marsden} for the compact case, Guediri \cite{Guediri} for  2-step nilpotent and Bromberg-Medina \cite{BM} for the Euclid  group (isometry group of the Euclidean plane). 

 The cases of Theorem \ref{main} give  alternative proofs of the first two items and a proper extension of the third one. The last item is not trivial, since an embedded submanifold of a complete Lorentzian manifold which is closed as a subset may be incomplete (an example can be constructed by using a spacelike curve in Lorentz-Minkowski which converges to a lightlike line so fast that its length is finite). It is not clear whether, in the case of left-invariant metrics on Lie groups, a general result on completeness for subgroups is possible.
  
 For the sake of completeness of exposition, we also prove that $k$-step nilpotent groups are not of linear growth for $k\geq 3$ (thus arriving at Cor. \ref{cor:k-step-nilpotent}) and that the groups admitting an idempotent have exponential growth along a curve (cf. Prop. \ref{prop:idemp-exp-growth}). This is consistent with the existence of incomplete left-invariant metrics on these groups, which has been known since \cite{Guediri} and \cite{BM}, respectively.
   
\subsection{The affine group of the line and prospects}   As an application of our techniques, we present a detailed study of the group $\mathrm{Aff}^+(\R)$ of orientation preserving affine transformations of the real line $\R$ in Section \ref{s6}. This group is bidimensional and it is not of linear growth. There are two  orbits by the action of Aut$(\mathfrak{g})$ of  Lorentzian left-invariant metrics, both incomplete, so:

\begin{corollary}\label{Cor:affine-group}
A left-invariant metric $g$ on $\mathrm{Aff}^+(\R)$ is complete if and only if $g$ is (positive or negative) definite.
\end{corollary}

 These results are consistent with the case of left-invariant metrics on the hyperbolic group, studied in \cite{Vuk-Suk}. Here, the aim of our study is twofold. First,  the explicit computations  illustrate the subtleties of our approach. 
 In particular, the  computation of the Clairaut metrics shows an announced behaviour: for two Wick rotated left-invariant metrics $g^{(\epsilon )}$, one of them Riemannian ($\epsilon=1$) and the other Lorentzian ($\epsilon=-1$), the Clairaut metrics 
$h^{(\epsilon)}, $   have the same growth, even if $h^{(1)}$ is complete and $h^{(-1)}$ is incomplete (and accordingly so are $g^{(1)}$ and $g^{(-1)}$). A heuristic reasoning and detailed proofs are provided in Sections \ref{s6.4} and \ref{s6.5}, respectively.    Other properties, shown explicitly, include the exponential growth associated with idempotents or the absence of bi-Lipschitz bounds for Clairaut metrics associated with different classes of incomplete Lorentzian metrics.  Second, we expect that such techniques might be applicable further.    
 What is more, this behaviour suggests open technical  questions, as:  (i) in which cases the completeness of the left-invariant metric $g$ implies the completeness of its Clairaut uniformity and (ii) to what extent the Clairaut uniformities of different  left-invariant metrics which are complete must be equal. 
   
From a more general viewpoint, the following questions are left open for a further study:   (a)~to what extent the list of groups of  linear growth is complete, (b)  whether a Lie group may or not admit a non-linear primary bound (see Remark \ref{r_log}) and (c) find  Lie groups with all its left-invariant metrics complete but not satisfying a primary bound.

\section{Clairaut metrics}\label{s2}

 We proceed with a detailed construction of the Clairaut metrics, remark how they allow us to transform a problem of semi-Riemannian geometry into a problem of Riemannian geometry and, furthermore, provide some abstract formalism.

 \subsection{The left-invariant metrics and  uniformity of $G$}\label{s2.1}
As any topological group, see \cite[Chap. II and III.3]{Bourbaki-a} and \cite[Chap. IX.1]{Bourbaki-b}, $G$ admits a natural left uniform structure whose  basis of entourages is 
$$
\{V_U: U \; \hbox{is a neighbourhood of 1}\} \qquad \hbox{where} \qquad
V_U:=\{(p,q)\in G\times G: \, q^{-1}.p\in U\}.
$$
Recall that any Lie group has a natural class of Riemannian metrics, the   left-invariant ones, which are complete.  In this respect, we can easily prove the following. 

\begin{lemma}\label{lem: left uniform}
Any two left-invariant Riemannian metrics $g_R,\,  {g'_R}$ on $G$ are bi-Lipschitz bounded, that is, there is a constant $c>0$ such that 
$$c\, g_R \leq {g'_R} \leq g_R/c.$$
Thus,  they induce equivalent distances  and the same uniformity on $G$. This uniformity is complete and equal to the left uniformity of $G$.
\end{lemma}
\begin{proof}
At 1, $g_R$ and ${g'_R}$ are scalar products, which are always bi-Lipschitz bounded (say, by suitably bounding the maximum and minimum eigenvalues of a naturally associated symmetric matrix), and such a bound is extended to $G$ by left-invariance. The completeness of the induced uniformity comes from the completeness of $g_R$   and its agreement with the left uniformity of $G$ follows by taking the neighbourhoods  $U$ above  as balls centered at 1 for $g_R$.
\end{proof}

\begin{remark}
If $g_R$ and $g'_R$ are bi-Lipschitz bounded then the identity map $(G, g_R)\longrightarrow (G,g'_R)$ is a quasi-isometry and, as consequence, a pseudo-isometry. Furthermore, in particular, $(G,g_R)$ and $(G,g'_R)$ are rough isometric in the sense of Kanai \cite{Kanai}.
\end{remark}

\subsection{The construction}
\label{s2.2} Let $G$ be a Lie group, with unit element 1,  endowed with a left-invariant semi-Riemannian metric $g$
 and let $(e_i)$ be  any vector basis  of the Lie algebra $\g  = T_1 G$. Take $X_i$ as the extensions of $e_i$ as left-invariant vector fields, and $Y_i$ as the extensions of $e_i$ as right-invariant vector fields (so $X_i(1)$ = $Y_i(1)$ = $e_i$). 
Since $g$ is left-invariant, all $g (X_i, X_j)$ are constant on $G$ and, moreover, the $Y_i$ are Killing fields.  We will adopt an obvious simplified notation for left and right translations $L_p$, $R_p$ and their tangent maps, so that  $X_i(p)= p.e_i, Y_i(p)= e_i .p$. 

Let $\omega^i$ be the  dual form of $Y_i$ with respect to $g$, i.e. $\omega^i = g( Y_i, \cdot )$. More precisely $\omega^i(u) = g( Y_i(p), u)$, for $u\in T_p G.$ We remark that the $(\omega^i)$ span the cotangent bundle $T^\ast G$ and that they are invariant under the geodesic flow of $g$, that is, if $\gamma(t)$ is a geodesic, then $\omega^i(\dot{\gamma}(t)) = \omega^i(\dot{\gamma}(0))$, cf. for example \cite[Lem. 9.26]{ONeill}.

\begin{definition}\label{d_Clairaut}
The 2-symmetric tensor $h = \sum \omega^i \otimes \omega^i$ 
is a Riemannian metric  on $G$ which 
 will be called  the {\em Clairaut metric}  associated to $g$  from the basis $(e_i)$. 
\end{definition}

We observe that $h$ is also a first integral of the geodesic flow of $g$. We remark, furthermore, that the Riemannian metric $h$ is not, in general, left-invariant. In fact, letting $u\in T_1 G$, it satisfies the following {\it transformation law} 
\begin{eqnarray}
\omega^i_p(p.u) & = & g_p( Y_i(p), p.u )  = g_p( e_i.p , p.u )   =   g_1( \Ad_{p^{-1}}( e_i), u ) \smallskip \label{eq: omega}\\ & = &g_1 ( e_i, ((\Ad_p)^{-1})^\dagger) (u) ) = \omega^i_1 ((\Ad_p^\dagger)^{-1} (u)). 
\label{eq: omega2}\end{eqnarray}
Here $\dagger$ denotes the formal adjoint of a linear map with respect to the semi-Riemannian metric $g$
(note that taking the inverse commutes with taking the formal adjoint). A similar computation shows that $h$ is also not right-invariant,
\begin{equation*} 
\omega^i_p(u.p) = g_p(e_i.p, u.p)=
g_1(R_p (e_i), R_p(u))=
g_1(R_p^\dagger R_p (e_i), u),
\end{equation*} 
unless $g$ is bi-invariant.
Finally, letting $u,v \in T_1 G$, from equation \eqref{eq: omega} it follows that
\begin{equation}\label{eq: h_explicit}
h_p(p.u, p.v)  =  \sum_i  \; g_1 ((\Ad_{p^{-1}})(e_i), u) \;  g_1 ((\Ad_{p^{-1}})(e_i),v). 
\end{equation}

\begin{proposition}\label{prop: Clairaut uniform} Let $(e_i)$, $(\hat e_i)$ be two vector bases of $T_1G$ and  $h, \hat h$ be the Clairaut metrics associated to $g$ from them, respectively. Then $h$ and $\hat h$ are bi-Lipschitz bounded and, in particular, induce uniformly equivalent distances on $G$. Moreover, if the transition matrix $M$  between $(e_i)$ and $(\hat e_i)$ is orthonormal then $h=\hat h$.
\end{proposition}

\begin{proof}
 The transition matrix $M=( M^i_j)$ satisfies $\hat e_j= \sum_i M^i_j e_i$. From \eqref{eq: omega2},
\begin{eqnarray*} \quad \hat\omega^j_p(p.u)  =  g_1 ( \hat e_j, ((\Ad_p)^{-1})^\dagger) (u) ) = \sum_k M^k_j 
g_1 (  e_k, ((\Ad_p)^{-1})^\dagger) (u)) 
=\sum_k M^k_j \, \omega^k_p(p.u).
 \end{eqnarray*}
Thus,
$$
 \sum_j (\hat \omega^j)^2 
= \sum_{k,l}\sum_j  M^k_j \omega^k M^l_j \omega^l = \sum_{k,l}  \omega^k (M^T\cdot M)^{kl} \, \omega^l 
$$ 
and the required inequality follows by taking a constant $c>0$  to suitably bound
the minimum and maximum eigenvalues of the (positive definite, symmetric) matrix $M^T\cdot M$. For the last assertion, use $M^T\cdot M=I$ (identity matrix).
\end{proof}
\begin{definition}\label{def: Clairaut unif} The set of bi-Lipschitz Clairaut metrics associated to $g$ will be called the {\em Clairaut class of $g$} and the corresponding
 unique uniformity  the {\em Clairaut uniformity} of $g$. 
 Eventually, any representative of the Clairaut class will be  called the {\em Clairaut metric of $g$} if the properties to be considered are independent of the choice.
\end{definition}

\subsection{Completeness}\label{s2.3} 

As a consequence of the fact that  $\omega^i(\dot\gamma(t))$ is a constant for any inextensible geodesic $\gamma$ of $g$ and, thus $h(\dot\gamma(t), \dot\gamma(t))$ is a constant,  one has the following fundamental result.

\begin{theorem}\label{c02} The left-invariant semi-Riemannian metric $g$ is geodesically complete if its associated Clairaut metric $h$ is complete. 
\end{theorem}

\begin{proof} The curve $\gamma$ restricted to any bounded interval $I\subset \R$ 
has finite $h$-length. Thus, by the completeness of $h$, $\gamma$ is continuously extensible to the closure of $I$ and, then, it is extensible as a geodesic  of $g$ (see  \cite[Prop. 3.28, Lem. 1.56]{ONeill}). 
\end{proof}    

 As a simple application, an alternative proof of a well-known result on completeness (see, for example,  \cite[Prop. 11.9]{ONeill}) is obtained.
\begin{corollary}\label{c_bi-invariant} Any bi-invariant semi-Riemannian metric $g$ on G is complete.
\end{corollary}
\begin{proof}
As seen above, each $\omega^i$  is right-invariant in this case and, thus, so is the  Clairaut metric $h$. As $h$ is Riemannian,  then it is  complete and Th. \ref{c02} applies to $g$.
\end{proof}

Another direct consequence is the following.

\begin{corollary}\label{cor:mardsen}
Any left-invariant semi-Riemannian metric $g$ on a compact Lie group G is complete.
\end{corollary}

\begin{proof}
The Clairaut metric associated to $g$ is complete by the compactness of G (use Hopf-Rinow theorem) and apply Th. \ref{c02}. 
\end{proof}

\begin{remark}\label{r_Marsden} Corollary \ref{cor:mardsen} is also a particular case of  Marsden's  theorem.  However,  
this theorem can be easily reobtained by extending the idea of Clairaut metrics to  homogeneous manifolds.  In this case, one does not have enough pointwise linearly independent Killing vector fields on $M$ to define an (everywhere non-degenerate) $h$ as above.
 However, for each  $p\in M$,  the homogeneity of $M$ permits to define a positive semidefinite metric $h_{p}$ on $M$ which is positive definite in a neighbourhood  $U\ni p$, namely, choose any basis $(e_j)$ of $T_pM$, extend it to Killing vector fields $Y^{(p)}_j$ and put $h_p = \Sigma \omega_{Y^{(p)}_j}^2$.  Using the compactness of $M$, the required complete Riemannian metric $h$ can be found as a finite sum $h=\sum h_{p_k}$.  
\end{remark}

Taking into account Th. \ref{c02} above,   our aim will be to prove the completeness of Clairaut metrics under appropriate hypotheses.

\subsection{Wick rotation} \label{s2.4}
 As shall be seen, a way to prove the completeness of the Clairaut class is to suitably bound it in terms of a left-invariant Riemannian one. Given the left-invariant semi-Riemannian metric $g$, we will choose a representative  $h$ of the  Clairaut class by using an orthonormal basis  $(e_i)$  of $g_1$ and, then, constructing a left-invariant Riemannian metric $\tilde g$ by imposing that $(e_i)$ is also orthonormal for $\tilde g_1$.
 We will say that such metrics $g$, $\tilde{g}$ are obtained  by {\em Wick rotation}. More precisely, let

\begin{equation*}
g_1( e_i, e_j) =\epsilon_i \delta_{ij}, \qquad \epsilon_i=
\left\{ 
\begin{array}{rl}
-1, & i=1, \dots, s \\
1, & i=s+1 \dots, n \end{array} \right. , \qquad 
\delta_{ij} = \left\{ \begin{array}{rl}
1, & i=j \\
0, & i\neq j\
\end{array}.  
\right.
\end{equation*}
Consider the  linear map $\psi$  such that $\psi(e_i)=\epsilon_i e_i$ and let $\tilde{g}$ be the left-invariant Riemannian metric such that  $\tilde{g}_1(e_i, e_j) = \delta_{ij} = g_1(e_i, \psi(e_j))$. Notice that $\psi$ is self-adjoint with respect to both $g_1$ and $\tilde{g}_1$.
Now, taking \eqref{eq: h_explicit} and using the positive definite scalar product $\tilde{g}_1$,  we have that

\begin{eqnarray} 
h_p(p.u, p.v)  & = &  \Sigma_i  \; \tilde{g}_1 ((\Ad_{p^{-1}})(e_i), \psi(u))\; \tilde{g}_1 ((\Ad_{p^{-1}})(e_i),\psi(v)) \smallskip \label{eq: h_explicit0}\\
& = &  \Sigma_i  \; \tilde{g}_1( e_i, (\Ad_{p^{-1}})^\ast (\psi(u))) \;   \tilde{g}_1( e_i, (\Ad_{p^{-1}})^\ast (\psi(v)))  \smallskip \nonumber \\
& = & \tilde{g}_1((\Ad_{p^{-1}})^\ast (\psi(u)),(\Ad_{p^{-1}})^\ast (\psi(v))),   \label{e: h previa}
\end{eqnarray}
where $^\ast$ is the formal adjoint (or transpose)  with respect to the positive definite scalar product $\tilde{g}_1$.  Summing up: 
\begin{proposition}\label{prop: Wick} Let $g, \tilde g$  be Wick rotated metrics and $h$ be the Clairaut metric constructed from one common orthonormal base $(e_i)$. Then $h$ is unique (independent of the chosen $(e_i)$) and
\begin{equation*}
h_p(p.u, p.v) = \tilde{g}_1(\Ad_{p^{-1}}^\ast (\psi(u)),\Ad_{p^{-1}}^\ast (\psi(v))) 
\end{equation*}
\end{proposition}    
\begin{proof} Use   \eqref{e: h previa} and, for the uniqueness of $h$,  the last assertion in Prop. \ref{prop: Clairaut uniform}. 
\end{proof}

\begin{remark}\label{r: Wick} We have, in particular, that if the left-invariant metric $g$ is Riemannian, then the Wick rotated metric is $\tilde g=g$. Notice that even in this case the expression of the Clairaut metric $h$ is non-trivial (see Remark \ref{r_BEWARE}).  
\end{remark}

\subsection{Abstract formalism}\label{s_Abstractformalism}\label{s2.5}
Let $\hbox{Sym}(\g)$ be the vector space of symmetric bilinear  forms  on $\g$ and $\hbox{Sym}^+(\g )$  the convex subset containing all the positive definite ones (i.e, the Euclidean scalar products). Choose   $m\in \hbox{Sym}^+(\g )$ and let $\hbox{SEnd}(\g,m)$ be the vector space of endomorphisms on $\g$ which are self-adjoint for $m$. Recall that there exists a natural isomorphism:
$$
\hbox{Sym}(\g) \rightarrow \hbox{SEnd}(\g,m), \qquad 
b\mapsto \Phi^b \quad \hbox{where} \quad b(u,v)= m(u,\Phi^b(v)), \quad \forall u,v\in \g .  
$$
Let  $\hbox{GL}(\g)$ be the group of vector automorphisms of $\g$. Any $A\in $ $\hbox{GL}(\g)$ induces a map
$$
\hbox{Sym}(\g )\rightarrow \hbox{Sym}(\g ), \qquad m\mapsto A.m
\quad \hbox{where} \quad (A.m)(u,v)= m(A^{-1}(u),A^{-1}(v)), \quad \forall u,v\in \g
$$ 
(this naturally restricts to a map
$\hbox{Sym}^+(\g )\rightarrow \hbox{Sym}^+(\g )$). 
Any representation of $G$ in $\g$ permits to extend each $m\in \hbox{Sym}^+(\g )$ to a Riemannian metric on $G$ (as well as any $b\in \hbox{Sym}(\g )$ to a  semi-Riemannian or degenerate one). Explicitly for the adjoint representation  $\Ad: G\rightarrow $  $\hbox{GL}(\g)$:
$$
\hbox{Sym}^+(\g ) \rightarrow \hbox{Riem} (G), \qquad m\mapsto \Ad.m \qquad \hbox{where} \quad (\Ad.m)_p (p.u,p.v)= m(\Ad_{p^{-1}}(u),\Ad_{p^{-1}}(v)),
$$
for all $u,v\in \g$,  $p\in G$ 
(notice that the trivial representation, $p\mapsto $ Identity, yields the left-invariant metrics).

In this abstract framework, the Clairaut metrics are constructed as follows. Given any left-invariant metric $g$ on $G$ and basis $(e_i)$ of $\g$, consider the positive scalar product $m$ such that $(e_i)$ is an $m$-orthonormal basis. Put $b=g_1$ and take the $m$-self-adjoint endomorphism
$\Phi^{b=g_1}$, such that 
$$g_1(e_i,e_j)= m(e_i, \Phi^{g_1}(e_j))$$ for all $i,j$. This is regardless of whether $(e_i)$ is $g_1$-orthonormal or not and $\Phi^{g_1}$ is $g_1$-self-adjoint or not.  Then, formula \eqref{eq: h_explicit} still follows from \eqref{eq: h_explicit0} putting $\psi=\Phi^{g_1}$, i.e.:
\begin{eqnarray*}
h_p(p.u, p.v) & = & \Sigma_i  \; m ((\Ad_{p^{-1}})(e_i), \Phi^{g_1}(u))\; m((\Ad_{p^{-1}})(e_i),\Phi^{g_1}(v)) \\  & = & m(\Ad_{p^{-1}}^\ast \circ\Phi^{g_1}(u),\Ad_{p^{-1}}^\ast \circ\Phi^{g_1}(v)) 
\end{eqnarray*}
where $^\ast$ is the $m$-adjoint. Using that $\Phi^{g_1}$ is a linear isomorphism, 
$\Ad_{p^{-1}}^\ast \circ\Phi^{g_1}=
\left((\Phi^{g_1})^{-1}\circ\Ad_{p}^\ast \right) ^{-1}$ and
\begin{equation}\label{e: abstract final}
h_p= \left((\Phi^{g_1})^{-1}\circ\Ad_{p}^\ast \right) .m. 
\end{equation}

\begin{remark} Regarding the abstract  formula \eqref{e: abstract final}, we observe the following:

\begin{itemize}
\item[(1)] The consistency in the use of $m$ above comes from the uniqueness of $h$ in Prop. \ref{prop: Wick}. 

\item[(2)] The change of $m$ by a different Euclidean scalar product $\hat m$ corresponds to a change of basis $(e_i)\leadsto~(\hat e_i)$, which was shown to imply bi-Lipschitz bounded Clairaut metrics in  Prop. \ref{prop: Clairaut uniform}. The previous
formula \eqref{e: abstract final} permits to identify where the bi-Lipschitz bounds come from. 
 
 \item[(3)] However, the change of $\Phi^{g_1}$ (i.e. of the left-invariant metric $g$) do {\em not} lead to bi-Lipschitz bounds, as will be stressed in Remark \ref{r_BEWARE} and shown explicitly in \S \ref{s6}. 
 
 \item[(4)] By using different combinations of the 1-forms $\omega^i$ in Def. \ref{d_Clairaut}, one can construct norms and Clairaut Finsler metrics as, for example,  
 Max$_i |\omega^i|$, 
 $\sum_i |\omega^i|$ or $\sqrt[4]{\sum_i (\omega^i)^4}$  (the two first examples  might be simpler to compute in some cases). Notice that, 
using  \eqref{eq: omega}, a Finslerian
combination as the previous ones would lead to an expression similar to \eqref{eq: h_explicit} and, then, a bi-Lipschitz bound as in Prop. \ref{prop: Clairaut uniform} would also hold. Such a metric would be also a first integral of the geodesics and, following the proof of  Th.~\ref{c02}, the completeness of a Finsler Clairaut metric also  implies the completeness of $g$ (the $g$-adjoint $\dagger$ can also be extended directly to these Finsler metrics).
\end{itemize}
\end{remark}

We can consider also the action of the group  $\mathrm{Aut}(\g) < \mathrm{GL}(\g)$ containing the Lie algebra automorphisms  of $\g$. $\Aut(\g)$ acts on $\mathrm{Sym}(\g)$ 
(as a restriction of the action of GL$(\g )$ above) and the next result  summarizes  the properties of the  orbits of the non-degenerate bilinear forms in $\mathrm{Sym}(\g)$.
Recall, by Lie's second theorem, that Aut$(\g)$ is in a natural one-to-one correspondence with the Lie group  automorphisms of the universal cover of $G$.

\begin{proposition}\label{p_aut_orbits} Let $G$ be a connected  Lie group, $g$ a left-invariant semi-Riemannian metric,  $\varphi \in \mathrm{Aut}(\g)$ and $g^\varphi$ the left-invariant metric such that $(g^\varphi )_1= \varphi . g_1.$ (i.e.,  the 
 $\varphi$-pushforward  of $g_1$). Then:

(1) The metric $g^\varphi$ is complete if and only if so is $g$. 

(2) All the left-invariant semi-Riemannian metrics in each orbit of $\mathrm{Sym}(\g)$ by the action of $\mathrm{Aut}(\g)$ are either complete or incomplete.

(3) All the Clairaut metrics associated to  left-invariant metrics on the same orbit are bi-Lipschitz bounded.
\end{proposition}

\begin{proof} 
(1) The completeness of $G$ with a left-invariant metric is equivalent to the completeness of its universal covering, so, we can assume that $G$ is simply connected. Then  $\varphi$ can be regarded as the differential at $1$ of a Lie group automorphism $\rho$  which, regarded as a map between semi-Riemannian manifolds  $\rho: (G,g)\rightarrow  (G,g^\varphi) $, is an isometry. Thus,   a curve $\gamma$ on $G$ will be a  complete geodesic with respect to  $g$  if and only if so is    $\rho\circ\gamma$  with respect to  $g^\varphi$.

(2) Straightforward from (1) as, trivially, the action is  transitive on each orbit.

(3)  Consider an orthonormal basis $(e_i)$ of $\g$ with respect to $g_1$. Let $\tilde g_1$ be the Wick rotated Euclidean scalar product and $\psi$ the self-adjoint map such that $\tilde g_1 (\cdot, \cdot) = g_1( \cdot, \psi \cdot)$. Let $\omega^i$ and $(\omega^\varphi)^i$ be the Clairaut first integrals obtained from $(g, (e_i))$ and $(g^\varphi, (e_i))$, respectively. We have that 
$$(\omega^\varphi)^i_p(p.u)  = \tilde g_1(\varphi^{-1} (\Ad_{p^{-1}} (e_i)), \psi \varphi^{-1}(u))  = \tilde g_1( \Ad_{p^{-1}}(e_i), (\varphi^{-1})^\ast \psi \varphi^{-1}(u)).$$
Since $(\varphi^{-1})^\ast \psi \varphi^{-1}$  is $\tilde g_1$-self-adjoint, $\tilde g_1(\cdot, (\varphi^{-1})^\ast \psi \varphi^{-1}(\cdot))$ is thus another scalar product and so is $\tilde g_1 (\cdot, \psi \cdot)$. Therefore, since all scalar products are equivalent, we have that $(\omega^\varphi)^i$ is bi-Lipschitz bounded to $\omega^i$, since $\omega^i(p.u) = \tilde g_1 (\Ad_{p^{-1}}(e_i), \psi(u))$. Thus, we can conclude that the Clairaut metrics associated to $g$ and $g^\varphi$ are bi-Lipschitz bounded.
\end{proof}

\section{Completeness under growth bounds}\label{s3}
\subsection{General result on preservation of completeness}
                                                                                                                                                                                                                                                                                                                                                                                                        In order to find a natural sufficient condition for the completeness of the Clairaut metrics, recall first the following general result.

\begin{proposition}\label{p_03}
Let $(M,g_R)$ be any (non-compact) connected complete Riemannian manifold, choose $x_0\in M$ and let $M\ni x \mapsto d_R(x)$ be the corresponding distance function from $x_0$. 
 Let 
 $\varphi: [0,\infty[ \rightarrow  ]  0,\infty[$ be any locally Lipschitz function which is {\em primarily complete}  i.e., satisfying
\begin{equation} \label{e: phi}
 \int_0^\infty \frac{1}{\varphi(r)}dr= \infty.
\end{equation} 
If $h$ is a Riemannian metric on $M$ with pointwise norm $\Vert \bdot \Vert_h$  satisfying:
$$ \Vert v_x\Vert_h \geq \frac{\Vert v_x \Vert_{R}}{ \varphi(d_R(x)) },  \qquad  \forall x\in M, v_x\in T_xM,$$
then $h$ is complete. In this case, $\Vert \bdot \Vert_h$ is said  to be {\em lowerly primarily bounded} with respect to the (primarily complete) {\em bounding  function} $\varphi$.

In particular,  this occurs when the bounding function $\varphi$ grows at most linearly,  that is, $\varphi(r)\leq 
a + b\, r$ for some constants $a, b \geq 0$, so that 
\begin{equation}\label{e: phi_lin}
 \Vert v_x\Vert_h \geq \frac{
\Vert v_x\Vert_{R}}
{a +b\, d_R(x)}, \qquad  \forall x\in M, v_x\in T_xM.
\end{equation}
\end{proposition}

\begin{proof}
 Recalling Hopf-Rinow theorem, it is enough to check that any (regular, piecewise-) smooth  curve $\gamma$ on $M$ starting at $x_0$ which is diverging (i.e., its image is not included in any compact subset of $M$) must have infinite $h$-length. 

 With no loss of generality, we can assume that $\gamma$ is reparametrized with respect to the 
arclength for $g_R$, thus  $\gamma: [0,\infty) \rightarrow M$,  by the completeness of $g_R$. The function $d_R$ is $g_R$-Lipschitz everywhere and   smooth outside the cut locus from $x_0$. The composition   $d_R\circ \gamma$ is $g_R$-Lipschitz  with constant 1 too, as:
\begin{equation*}
|(d_R\circ \gamma) (t+s)- (d_R\circ \gamma) (t)| \leq d_R(\gamma (t+s), \gamma (t)) \leq \hbox{length}_R(\gamma|_{[t,t+s]}) = |s|, 
\end{equation*}
the first inequality following from the triangle one. By Rademacher theorem  $r:= d_R\circ \gamma$  is differentiable almost everywhere and, then, its derivative $\dot r$  satisfies
\begin{equation}\label{er}
\dot r (t) \leq 
|\dot r(t)| \leq 1= \Vert \dot \gamma(t)\Vert_R \leq \varphi(r(t)) \Vert \dot \gamma(t) \Vert_h
\qquad \qquad \hbox{a.e. for } \; t\in [0,\infty).
\end{equation}
As $r(0)=0$ and $\lim_{t\rightarrow \infty} r(t)=\infty$ (recall that $g_R$ is complete),  
$$
\hbox{length}_h(\gamma) = \int_0^\infty 
\Vert \dot \gamma(t) \Vert_h dt \geq
\int_0^\infty 
 \frac{\dot r(t)}{\varphi(r(t))} dt 
= \int_0^\infty 
 \frac{dr}{\varphi(r)} =\infty, 
$$
(the inequality by \eqref{er} and, then,  a change of variable in an interval), as required.
 \end{proof}

\begin{remark}\label{r_log} 
 (1)  Similar type of bounds  have been known for several purposes on completeness since the seventies (as explained in the classical  Abraham-Marsden's book, \cite[p. 233]{AbMarsden}) 
 and applications to the trajectories of many mechanical systems are known also, \cite{CRSanchez}. In our case, a more straightforward approach is used  and a simple self-contained  proof adapted to our case is provided  (compare with  the review \cite{Sanchez-survey}).   

 (2) With more generality, $\varphi$ is primarily complete not only when it grows at most linearly but also whenever $\varphi(r)\leq r (\log r)(\log  (\log r))...(\log^k(\dots    (\log r)))$ for large $r$ and some $k\in \N$. Next, we focus on the at most linear hypothesis because this  is  natural in the framework of Lie groups (as they are analytic), however, the general primary bound could be also used in what follows.   
\end{remark}

\subsection{Non-preservation of uniformities}

Notice that, for $\varphi$ satisfying \eqref{e: phi}, the inequalities
\begin{equation}
\label{r: no_unif}
 \frac{\Vert v_x\Vert_R}{ \varphi(d_R(x)) } \leq \Vert v_x\Vert_h \leq  \varphi(d_R(x))\Vert v_x\Vert_R ,  \qquad  \forall x\in M, v_x\in T_xM, \end{equation}
imply that $g_R$ is complete if and only if so is $h$; in this case, $g_R$ and $h$ will have the same Cauchy sequences. However, even in this case the uniformities of $g_R$ and $h$ may be different as, in particular, the following example shows.

\begin{example}\label{Ex_Unpreserved uniformities} [Unpreserved uniformities]  Consider $R=dz^2$ and $h=dz^2/(1+z^2)$,  $z\in \R$ (this is a particular case of \eqref{e: phi_lin} and  \eqref{r: no_unif}). These metrics are both complete, they  have the same Cauchy sequences and the uniformity of $h$ is less fine than the one of $R$ (as $d_h\leq d_R$). The converse would hold if and only if for any $\epsilon>0$, there exists $\delta>0$ such that if $d_h(x,y)<\delta$ then $d_R(x,y)<\epsilon$ (for this criterion, see \cite[Chap. IX.3]{Bourbaki-b}). We will find a contradiction for, say, $\epsilon=1$ as follows. Recall $\int dx/\sqrt{1+x^2}=$ arcsinh$(x)$, thus, if $0<x<y$: 
$$d_h(x,y)= \hbox{arcsinh}(y)-\hbox{arcsinh}(x).$$ Taking $\sinh$ on both sides one has: 
$$\begin{array}{rl}
\sinh(d_h(x,y))= &  y \sqrt{1+x^2}-x\sqrt{1+y^2}=
\frac{y^2-x^2}{y \sqrt{1+x^2}+x\sqrt{1+y^2}}=(y-x)\frac{y+x}{y \sqrt{1+x^2}+x\sqrt{1+y^2}}
\\ 
=& d_R(x,y) \frac{y+x}{y \sqrt{1+x^2}+x\sqrt{1+y^2}}. \end{array}
$$
Now, choose any $\delta>0$ and take, for each $x>0$, $y(x):=x+ \delta(x)$ with $\delta(x)>0$ such that $d_h(x,y(x))=\delta/2$ (such a $\delta(x)$ exists by the completeness of $h$). Solving for $d_R(x,y=y(x))$ above:
$$\begin{array}{rl}
d_R(x,y(x) )= & \sinh(\delta/2)\frac{y(x) \sqrt{1+x^2}+x\sqrt{1+y^2(x)}}{y(x)+x} \geq
\sinh(\delta/2)\frac{ (x+ \delta(x))x +x(x+\delta(x))}{y(x)+x} 
\\ = & 
\sinh(\delta/2)\frac{2x+ 2 \delta(x)}{2x+\delta(x)}x \geq \sinh(\delta/2) x,
\end{array}
$$
and the last term diverges when $x\rightarrow \infty$ (taking, in particular, values greater than $\epsilon=1$).
\end{example}

\section{Groups of linear growth}\label{s4}

\subsection{Development of the notion} Let 
 $g_R$ be a left-invariant Riemannian metric on a Lie group $G$. Consider the map $r: G \longrightarrow \R$ where $r(p) = d_R(1, p)$ and $d_R$ is the distance induced by the Riemannian metric $g_R$, more precisely 
\begin{equation*}
r(p) = d_R(1,p) =  \inf_c \left\{ \int_{t_0}^{t_1} \Vert \dot{c}(t)\Vert_R  \, dt \, | \, c \text{ curve}, \, c(t_0) =1, \ c(t_1) = p \right\}.
\end{equation*}

\begin{lemma}\label{l_rg_rg-1}
For a left-invariant Riemannian metric $g_R$, we have that $r(p) = r(p^{-1})$, for every $p\in G$.
\end{lemma}

\begin{proof}
The statement is immediate from the fact that each left translation $L_p$ is a $g_R$-isometry; hence, $d_R(1, p^{-1}) = d_R(L_p(1), L_p(p^{-1})) = d_R(p,1)$.
\end{proof}

\begin{definition}[linear growth]\label{def:glg}
We say that a Lie group $G$ has  (at most)  linear growth if there exist a left-invariant Riemannian metric $g_R$ on $G$, a Euclidean scalar product with norm $\Vert \bdot \Vert$ on $\g$ and constants $a, b$ such that 
\begin{equation}
\dfrac{\Vert u \Vert}{a+b\, r(p) } \leq \Vert \Ad_p (u)\Vert \leq (a+b\, r(p)) \Vert u \Vert \label{eq:slg}
\end{equation}
for every $p\in G$ and for every $u\in \g$.
\end{definition}

 \begin{remark}\label{r_3 norms} If a  Lie group satisfies the linear growth condition for a left-invariant Riemannian metric $g_R$ and a  scalar product norm $\Vert \bdot \Vert$ then it satisfies this condition  with respect to any other   ${g'_R}$ and  $\Vert \bdot \Vert'$. Indeed, this occurs taking into account that there exists $c>0$ such that 
$$c \Vert \Ad_p (u) \Vert \leq \Vert \Ad_p (u) \Vert' \leq \Vert \Ad_p (u) \Vert/c , \quad 
c \Vert u \Vert \leq \Vert u \Vert'\leq \Vert u \Vert/c, \quad   c \, r(p) \leq  r'(p) \leq  r(p)/ c
,  $$ 
and using \eqref{eq:slg}.
 What is more, linear growth can be defined analogously with respect to any left-invariant norm. So, with no loss of generality, we will choose a left-invariant Riemannian metric $g_R$ and  $\Vert \bdot \Vert$ will be the norm associated to $(g_R)_1$. 
\end{remark}

\begin{proposition}\label{p_linearg} 
For a Lie group $G$, the following conditions are equivalent:
\begin{enumerate}
\item $G$ has linear growth.
\item There exist constants $a, b$ such that $\dfrac{\Vert u \Vert}{a+b\, r(p) } \leq \Vert \Ad_p (u)\Vert$, for every $p\in G$, for every $u\in\g$.
\item There exist constants $a, b$ such that $\Vert \Ad_p (u)\Vert \leq (a+ b\, r(p))\Vert u \Vert$, for every $p\in G$, for every $u\in\g$. 
\end{enumerate}
\end{proposition}

\begin{proof}
It suffices to show that $(2)$ and $(3)$ are equivalent. Suppose that $(2)$ holds. Then we have
\begin{equation*}
\Vert u \Vert = \Vert \Ad_{p^{-1}}(\Ad_p (u))\Vert \geq \dfrac{\Vert \Ad_p(u) \Vert}{a + b\, r(p^{-1})} 
\end{equation*}
and using the fact that $r(p) = r(p^{-1})$ then we get that $(3)$ holds. Proving the converse is analogous.
\end{proof}


 Next, let us study the meaning of linear growth and, with  more generality, the bounds on $\Ad_p$, $p\in G$.  Where there is no possibility of confusion,  $\Vert \bdot \Vert$ will also denote the supremum norm on operators, that is, 
\begin{equation}
\label{e:lambda+}
\Vert \Ad_p \Vert= \hbox{Max}_{\Vert u \Vert=1}\{ \Vert \Ad_p(u) \Vert \}= \lambda_+(p),
\end{equation}
where $\lambda_+(p)>0$ (resp. $\lambda_-(p)>0$) denotes the maximum (resp. minimum) of 
$$\{\sqrt{\Lambda_i}: \Lambda_i \, \, \hbox{is a eigenvalue of } \Ad_{p}^{\ast} \circ \Ad_{p}\}
$$
($^\ast$ denotes the $g_R$-adjoint).
As  the eigenvalues of $(\Ad_{p}^{\ast} \circ \Ad_{p})^{-1}$ are $\Lambda_i^{-1}$ one also has
\begin{equation}\label{e:lambda-}
\Vert \Ad_{p^{-1}}\Vert= 1/\lambda_-(p), \qquad \hbox{and using \eqref{e:lambda+},} \quad \Vert \Ad_{p}\Vert \; \Vert \Ad_{p^{-1}}\Vert \geq 1.
\end{equation}

These considerations can be summarized as follows.

\begin{proposition}
Let $\rho: G\rightarrow ]0,\infty[$, $p\in G$, 
\begin{enumerate}
\item The following  upper bounds are equivalent: 
\\
(i)  $\Vert \Ad_p (u)\Vert \leq \rho(p) \Vert u \Vert$ for all $u\in \mathfrak{g}$, (ii) $\lambda_+(p)\leq \rho(p)$, and (iii)~$\Vert \Ad_p \Vert \leq \rho(p)$.
\item The following  lower bounds are equivalent: 
\\
(i)  $ \Vert u \Vert/\rho(p)\leq  \Vert \Ad_{p} (u)\Vert$ for all $u\in \mathfrak{g}$, (ii) $ 1/\rho(p)\leq \lambda_-(p)$, and (iii)~$\Vert \Ad_{p^{-1}} \Vert \leq  \rho(p)$.
\end{enumerate}
Moreover, if $\rho(p)=\rho(p^{-1})$, for all $p\in G$, then the upper and lower bounds are equivalent.
\end{proposition}

\subsection{Linear growth of Clairaut metrics and  
proof of completeness}
 
\begin{lemma}\label{l_essential}
Let $G$ be a Lie group with linear growth. Then, the Clairaut metric $h$ associated to any pair of Wick rotated semi-Riemannian metrics $(g, \tilde g)$ satisfies  inequality \eqref{e: phi_lin} of Prop.~\ref{p_03} with respect to the Riemannian metric $g_R=\tilde g$.
\end{lemma} 
\begin{proof}
 Using Prop. \ref{prop: Wick},
\begin{equation}
\label{e_sharp1}
\begin{array}{rl}
h_p(p.u, p.u) = & 
\tilde{g}_1(\Ad_{p^{-1}}^\ast (\psi(u)),\Ad_{p^{-1}}^\ast (\psi(u))) \\ = &  \tilde{g}_1( \psi(u),\Ad_{p^{-1}} \circ \Ad_{p^{-1}}^\ast (\psi(u))) \geq 
\lambda_-(p^{-1})^2 \; 
\tilde{g}_1( \psi(u), \psi(u)) 
\end{array}
\end{equation}
where, in the last equality, we have used the definition of $\lambda_-$ below \eqref{e:lambda+}. Then, using that $\psi$ is an isometry for $\tilde g$ and \eqref{e:lambda-},
\begin{equation}
\label{e_sharp2}
h_p(p.u, p.u) \geq \frac{\tilde{g}_1( u, u)}{\Vert \Ad_{p}\Vert^2}.
\end{equation}
Now, taking into account the linear growth of $G$, we have $\Vert \Ad_{p}\Vert\leq a+ b\, r(p) $, for some constants $a,b$,
and thus:
$$\Vert p. u\Vert_{h}=
\sqrt{h_p(p.u, p.u)} \geq \frac{\sqrt{\tilde{g}_1( u, u)}}{\Vert \Ad_{p}\Vert}\geq 
\frac{\sqrt{\tilde{g}_p(p. u,p. u)}}{a + b\, r(p)}=\frac{\Vert p. u\Vert_{\tilde g}}{a + b\, r(p)},
$$ 
as required.
\end{proof}
\begin{theorem}\label{t_lg}
All the left-invariant semi-Riemannian metrics of a Lie group with linear growth are geodesically complete.
\end{theorem}
\begin{proof}
Given the left-invariant metric $g$, we can choose any of its orthonormal bases and construct the corresponding  Wick rotated metric $\tilde g$ and the Clairaut metric $h$. Lem. \ref{l_essential} permits to apply  Prop. \ref{p_03} and ensure the completeness of $h$. Therefore  Thm.  \ref{c02} applies and $g$ is necessarily  complete.
\end{proof}
 \begin{remark}\label{r_BEWARE} The inequalities \eqref{e_sharp1} and \eqref{e_sharp2} are sharp, as they are attained when $\psi(u)\in \mathfrak{g}$ is an eigenvector associated to the eigenvalue $\lambda_-(p^{-1})^2$ of $\Ad_{p^{-1}} \circ (\Ad_{p^{-1}})^\ast$. 
 As a consequence, the following noticeable  properties occur: 
 
 (1)  Fix a basis  $(e_i)$ of $\mathfrak{g}$ and consider all the left-invariant metrics of any signature such that this basis is orthonormal; in  particular,  all of them have the same Wick rotated Riemannian $\tilde g$. Then, {\em  the  same bound \eqref{e_sharp2} of the corresponding Clairaut metrics with respect to $\tilde g$ holds, with independence of the signature}. 
 
 (2)  However, if linear growth is not imposed, the completeness of these Clairaut metrics do depend on the signature (and explicit example will be given in \S \ref{s6}). The underlying reason is that the $\tilde g$ isometry $\psi$ does depend on the signature and, thus, the {\em direction} of the eigenvector corresponding to $\lambda_-(p^{-1})$ also depends on  the signature.  Intuitively, if the velocity of a curve $\gamma$ lies in the $\lambda_-$ direction its length is ``small'' and, thus, if $\gamma$ diverges its length may be finite and yield incompleteness. However, if $\gamma$ remains in a compact subset then completeness can occur. In particular, the behaviour of the Clairaut  metric is not trivial even when $g=\tilde g$, as anticipated in Remark \ref{r: Wick}.

 \end{remark}

\section{ Groups of linear growth in  Theorem \ref{main} } \label{s5}

In this section, we will prove that each  class of Lie groups in the list of  Theorem \ref{main} has linear growth which, in particular,  completes its proof.  Moreover, the non-linear growth of some related groups will be remarked upon.

\subsection{Bounded growth: abelian and compact groups}

We start by showing that abelian and compact groups not only have linear growth but, what is more, they have bounded growth.

\begin{proposition}\label{prop:growth-abelian}
If $G$ is an abelian Lie group, then $G$ has linear growth.
\end{proposition}

\begin{proof}
This is immediate from the fact that $\Ad_p = \Id$, for all $p\in G$. It suffices to take $a = 1$ and $b = 0$ in Def. \ref{def:glg}.
\end{proof}

\begin{proposition}\label{prop:growth-compact}
If $G$ is a compact Lie group, then $G$ has linear growth.
\end{proposition}

\begin{proof}  As the adjoint map $\Ad$ is $C^\infty$ (see for example \cite[Th. 3.45]{Warner}), its composition with the norm $\Vert \bdot \Vert$ is continuous and, thus, the map $p\mapsto \Vert  \Ad_p \Vert$  attains a maximum (and a minimum), by the  compactness of $G$.
 For the supremum norm, we also have $\Vert \Ad_p (u) \Vert \leq \Vert \Ad_p \Vert \Vert u \Vert$, for every $u\in\g$. Therefore, $G$ has linear growth with $a = \mathrm{max}{\Vert \Ad_p \Vert}$ and $b = 0$.
\end{proof}

\subsection{Subgroups and direct products} We will now show that the linear growth condition is closed under taking subgroups and direct products.

\begin{proposition}
Let $G$ and $H$ be Lie groups such that $H$ is a subgroup of $G$. If $G$ has linear growth then so does $H$.  
\end{proposition}

\begin{proof}
Let $g_R^G$ be a Riemannian metric on $G$ and $g_R^H$ be the restriction of $g_R^G$ to $H$. Fix $p\in H$, clearly $d_{G}(1, p) \leq d_{H}(1, p)$. Moreover, since the Lie algebra $\mathfrak{h}$ of $H$ is a Lie subalgebra of the Lie algebra $\g$ of $G$, then $\Vert \Ad^H_p \Vert \leq \Vert \Ad^G_p \Vert$. Thus, for some constants $a,b$, we have $\Vert \Ad_p^H \Vert \leq \Vert \Ad_p^G \Vert \leq a + b\, d_G(1,p) \leq a + b\, d_H(1,p),$ therefore proving that $H$ has linear growth.  
\end{proof}

\begin{proposition}\label{prop:direct-products}
If $K$ and $N$ are Lie groups with linear growth, then their direct product $K\times N$ also exhibits linear growth. 
\end{proposition}

\begin{proof}
Let $\mk$ and $\n$ be the Lie algebras of $K$ and $N$, respectively.  Consider left-invariant Riemannian metrics $g_R^K$ on $K$ and $g_R^N$ on $N$, with induced norms $\Vert \bdot \Vert_{\mk}$ on $\mk$ and $\Vert \bdot \Vert_\n$ on $\n$, and denote $r^K(k) = d_K(1_k, k)$, $r^N(n) = d_N(1_N, n)$, the induced distances to the identity of $K$ and $N$, respectively. From the linear growth condition on $K$ and on $N$, there exist constants $a_K, b_K, a_N, b_N$ such that, for every $k\in K$ and $n\in N$,
\begin{equation*}
\Vert \Ad_k u \Vert_\mk \leq  (a_K + b_K \, r^K(k)) \Vert u \Vert_\mk \quad \text{ and } \quad \Vert \Ad_n v \Vert_\n \leq  (a_N + b_N\, r^N(n)) \Vert v \Vert_\n  
\end{equation*}
for all $u\in \mk $, for all $v\in \mathfrak{n}$. Taking the (left-invariant) product Riemannian metric on $K\times N$  with induced norm $\Vert \bdot \Vert_{\mk\times \n}$ on  $\mk \times \mathfrak{n}$, we have
\begin{equation*}
\Vert \Ad_{(k,n)} (u,v) \Vert_{\mk\times\n} \leq \Vert \Ad_k u \Vert_\mk  + \Vert \Ad_n v \Vert_\n  \leq (a_K + b_K \, r^K(k)) \Vert u \Vert_\mk + (a_N + b_N\, r^N(n)) \Vert v \Vert_n.
\end{equation*}
Considering the orthogonal projection on each factor, then 
$d_N(1_N, n) \leq d_{K\times N}((1_K, 1_N), (k,n))$ and $d_K(1_K, k) \leq d_{K\times N}((1_K, 1_N), (k,n))$, since the metric on $K\times N$ is the product of $g_R^N$ and $g_R^K$. Thus, $r^N(n) \leq r^{K\times N}(k,n)$ and $r^K(k) \leq r^{K\times N}(k,n)$. Therefore,
\begin{eqnarray*}
 \Vert \Ad_{(k,n)} (u,v) \Vert_{\mk\times\n} & \leq & (a_K + b_K \, r^{K\times N}(k,n)) \Vert u \Vert_\mk  + (a_N + b_N \, r^{K\times N}((k,n)) ) \Vert v \Vert_\n   
 \\ & \leq & ((a_K + a_N ) + (b_K + b_N) r^{K\times N}(k,n))\Vert (u  ,  v ) \Vert_{\mk\times \n},
\end{eqnarray*} 
yielding the result. 

\end{proof}

\subsection{2-step nilpotent groups}

\begin{proposition}\label{prop:growth-2-step}
If $G$ is a  2-step nilpotent Lie group, then $G$ has linear growth. 
\end{proposition}

\begin{proof}

Let $Z$ be the centre of $G$. We have that  $\pi: G \longrightarrow G/Z$ is a fibration and its fibres are naturally isomorphic to $Z$. Consider $\g$ and $\z$ the Lie algebras of $G$ and $Z$, respectively. We have that $\ker d\pi_1 = \z$, and for this reason we will call the elements of $\z$ the vertical vectors.  Take any left-invariant Riemannian metric on $G$, which induces a positive definite inner product on $\g$, and define $\p$ as the orthogonal complement of $\z$ in $\g$, which we will call the horizontal vectors. Remark that $\p$ can be identified with $T_1(G/Z)$.
The left-invariant Riemannian metric has a well-defined projection on $G/Z$, yielding $G/Z$ as a Riemannian homogeneous space (in fact, a Lie group since $Z$ is normal in $G$), and such that $d\pi: (\ker d\pi)^\perp \longrightarrow T(G/Z)$ is an isometry.  In other words, the map $\pi: G \longrightarrow G/Z$ is a Riemannian submersion.

Observe, in fact, that since $Z$ is the centre and $G$ is 2-step nilpotent, $G/Z$ is abelian,  more concretely isomorphic to  some $\R^d$ (endowed with a left-invariant 
Riemannian and so Euclidean metric). Indeed, if $\tilde{G}$ is the universal cover of $G$, then, its fundamental group  
$\pi_1(G)$ is contained in the centre of $\tilde{G}$, \cite[Chap. I.3, p. 47, Cor. 2]{OniVin}. Hence the quotient of $G$ by its centre is 
isomorphic to the quotient of $\tilde{G}$ by its centre. The last quotient is  a simply connected abelian Lie group.

Let us now recall some properties of   Riemannian submersions, (\cite[\S 7, p. 212-213]{ONeill}). First, they  are distance contracting: for any $p, q \in G$, $d_{G/Z}(\pi(p), \pi(q)) \leq d_G(p, q)$. 
Also, a geodesic $t \mapsto \gamma(t)$ in $G$ is horizontal if it is somewhere, and hence everywhere, orthogonal to the fibres. Its projection $t \mapsto \pi(\gamma)(t)$ is
a geodesic and $\pi $ establishes an isometry between $\gamma$ and $\pi\circ \gamma$.  In our case, all geodesics in the Euclidean space 
$G/Z$ are (globally) minimizing, and thus  all horizontal geodesics of $G$ are minimizing. That is, $d_G(\gamma(t), \gamma(s)) = \vert t - s \vert$  $[= 
d_{G/Z}(\pi(\gamma(t)), \pi(\gamma(s)))]$, for any $t, s$
(assuming $\gamma$ is arc-length parameterised).  

If $z \in Z$, and $\gamma(t) z$ is a point in the fibre of $\pi(\gamma (t))$, then,
$d_G(\gamma(t) z, \gamma (s) ) \geq \vert t - s \vert$. Indeed a geodesic joining $\gamma(t) z$ and $\gamma(s)$ 
projects to a curve joining $\pi(\gamma(t))$ and $\pi(\gamma(s))$, and hence (by the contracting character of $\pi$) has length greater than or equal to
$ \vert t - s \vert$.
 
Our aim now is to show that horizontal geodesics through 1 are in fact one-parameter subgroups of $G$. 
Consider such a geodesic $\gamma$,   with  $\gamma(0) = 1$ and $\dot{\gamma}(0) = a$, with $a \in \p$ (the orthogonal of $\z$). 
We will show that  $\gamma(t) = \exp (ta)$. Given any curve $\gamma$ in $G$, we can consider the associated curve $x(t) = \gamma^{-1}(t) \dot{\gamma}(t)\in\g$ and, by the Euler-Arnold theorem, \cite{Arn}, the geodesics are in one-to-one correspondence with the solutions of
\begin{equation}\label{eq:EA-2nil}
\dot{x}(t) = \ad_{x(t)}^\ast x(t)
\end{equation}
where $\ast$ stands for the formal adjoint with respect to the  (positive definite) inner product on $\g$.  The curve $\gamma(t) = \exp(ta)$ is a geodesic if and only if $x(t) = a$ is a solution of the Euler-Arnold equation \eqref{eq:EA-2nil}. Clearly, $\dot{x}(t) = 0$. Moreover, for any $b\in \g$, we get that $\langle \ad_a^\ast a, b \rangle = \langle a, \ad_a b \rangle $. Since $\g$ is 2-step nilpotent then $\ad_a b \in \z$ and since $a\in \p$ then $\langle a , \ad_a b \rangle = 0$. Thus, $x(t)$ satisfies eq. \eqref{eq:EA-2nil}. 


Let $p  \in G\backslash Z$,  we proceed to showing that there exists $z\in Z$ with $pz$ of the form $\exp(ta)$, with $a\in \p$ such that $\Vert a \Vert = 1$. There is a $w\in\g$ such that $p= \exp(w)$. Consider the decomposition $w= u+v$ with $u\in \p$ and $v\in \z$ and take $z=\exp(-v)$.  Since $v\in \z$, then $z\in Z$ and, by the Baker-Campbell-Hausdorff formula, \cite[Chap. II, \S 6.5, p. 161]{Bourbaki-c}, since $[w,v]=0$, we get $pz = \exp(w)\exp(-v) = \exp(u)$, so that our claim follows with $a = u/{\Vert u \Vert}$.

Since $\gamma(t) = \exp(ta)$ is a minimizing unitary geodesic, $r(pz) = d_G(1,pz) = |t|$.  From our discussion above, 
$r(p) = d_G(1, p) \geq d_G(1, pz) \geq \vert t \vert$.

Remark now that $\Ad_p = \Ad_{pz}$. Also, $\Ad_{\exp(ta)} = \exp^{\ad_{ta}} = \Id + t\, \ad_a$, since $G$ is 2-step nilpotent.   Let $\beta = \Vert \Id \Vert $ and $\alpha = \Vert \ad_a\Vert$.  Thus, for any $u\in \g$, $\Vert \Ad_p u \Vert = \Vert \Ad_{pz} u \Vert \leq (\alpha |t|+ \beta ) \Vert u \Vert$ = $(\alpha r(pz) + \beta)\Vert u \Vert   \leq (\alpha r(p)  + \beta) \Vert u \Vert $.  If $z\in Z$, then $\Ad_z = \Id$, and the same inequality holds.

 \end{proof}

\begin{proposition}\label{prop:growth-3-step}
If $G$ is an $m$-step nilpotent Lie group, with $m\geq 3$, then $G$ does not have linear growth.
\end{proposition}

\begin{proof} Fix a left-invariant Riemannian metric $g_R$ on $G$. There exists $k>2$, there exist $a,u \in \mathfrak{g}$ such that $\ad_a^{k-1}(u) \neq 0$, and $ad_b^k=0$ for any $b\in\mathfrak{g}$. We can assume that $\Vert a \Vert = 1$. Let $p=\exp(ta)$, $t>0$, then $d_R(1, p) \leq t$ (since $\gamma(s) = \exp(sa)$ is a curve joining 1 to $p$ with length $t$) and
 \begin{equation*}
  \Ad_p (u) = \exp(\ad_{ta}) u = u + t\, \ad_a u + \frac{t^2}{2} \ad_a^2 u + \cdots + \frac{t^{k-1}}{(k-1)!} \ad_a^{k-1} u.
 \end{equation*}
For large values of $t$, we obtain 
\begin{equation*}
 \frac{\Vert \Ad_p(u) \Vert}{t^{k-1}}  \geq
C + O(t)
\end{equation*}
where
$C=\Vert \ad_a^{k-1}(u)\Vert /
(k-1)!>0$
and $O(t) \rightarrow 0$
as $t \rightarrow \infty$ .
Thus, again for large values of $t$, we have
\begin{equation*}
\Vert \Ad_p(u)\Vert \geq
C t^{k-1} -1 \geq
C d_R(1,p)^{k-1} -1
\end{equation*}
which implies non-linear growth.
\end{proof}

Combining Props. \ref{prop:growth-abelian}, \ref{prop:growth-2-step} and \ref{prop:growth-3-step}, we obtain the following result which characterizes linear growth for nilpotent groups.

 \begin{corollary}\label{cor:k-step-nilpotent}
 If $G$ is an $m$-step nilpotent Lie group, then $G$ has linear growth if and only if $m\leq 2$.
 \end{corollary}

\begin{remark}
The proof of Prop.~\ref{prop:growth-3-step} shows  that, for an $m$-step nilpotent Lie group $G$, the adjoint representation $\Ad$ has growth bounded by the $(k-1)$ power, for some $k>2$, of the distance to $1\in G$ (up to additive and multiplicative constant). This is an obstruction to the existence of idempotents for $G$, as seen from Prop.~\ref{prop:idemp-exp-growth} below.
\end{remark}

\subsection{Semidirect products} We will be mainly interested in certain types of semidirect products as shall be explained shortly. For the general theory of semidirect products of Lie groups, we refer to \cite[Chap. III.1.4]{Bourbaki-c}.

Let $K$ be any Lie group and $V (\simeq\R^n)$ be a real vector space, and suppose that we have a representation $\rho: K \longrightarrow \GL(V)$. The semidirect product $K\ltimes_\rho V$ is defined by considering the group structure on $K\times V$ given by
\begin{equation*}
(k_1, v_1).(k_2,v_2) = (k_1k_2, v_1 + \rho(k_1)(v_2)). 
\end{equation*}
The Lie algebra of $K\ltimes_\rho V$ is then the semidirect product 
$\mk \ltimes_{\rho_\ast} V$, 
where $\rho_\ast: \mk \longrightarrow \mathfrak{gl}(V)$ is the derived representation of $\rho$, that is, the differential map of $\rho$ at $1_K$. Here $V$ is the abelian Lie algebra of dimension $n$. The adjoint representation of $K\ltimes_\rho V$ is seen to be given by
\begin{equation*}
\Ad_{(k, v)}(u, w) = (\Ad_{k}u, \rho(k)w - \rho_\ast (\Ad_{k}u)v)
\end{equation*}
and the bracket structure on $\mk\ltimes_{\rho_\ast} V$ can be written as
\begin{equation*}
[(u_1, w_1), (u_2, w_2)] = ( [u_1, u_2], \rho_\ast(u_1)w_2 -\rho_\ast(u_2)w_1).
\end{equation*}

For our purposes, it suffices to take faithful representations $\rho$ since $K\ltimes_\rho V$ is isomorphic to $K^\prime \times (K/K^\prime \ltimes_\rho V)$, where $K^\prime = \ker \rho$ (and Prop. \ref{prop:direct-products} applies). Assuming that $\rho$ is faithful then $K$ can be seen as a subgroup of $\mathrm{GL}(V)$ and $K\ltimes_\rho V$ thought of as a subgroup of the group of affine motions $\mathrm{GL}(n) \ltimes \R^n$ with multiplication given by the matrix product 
\begin{equation*}
\left(\begin{matrix}
\rho(k_1) & v_1 \\
0 & 1
\end{matrix}\right)
\left(\begin{matrix}
\rho(k_2) & v_2 \\
0 & 1
\end{matrix}\right)
=
\left(\begin{matrix}
\rho(k_1 k_2) & v_1+\rho(k_1)v_2 \\
0 & 1
\end{matrix}\right).
\end{equation*}

In the sequel, $K$ will be the product of a compact Lie group and a vector space. To shorten terminology, we will call such a group $K$ a \emph{pseudo-compact} Lie group. Recall that a Lie group is pseudo-compact if and only if it can be equipped with a bi-invariant Riemannian metric, \cite{Milnor}. 

The following lemma is a general result on Riemannian metrics of semidirect products $K\ltimes H$ where $K$ is taken to be pseudo-compact but $H$ is not necessarily a vector space. The case where $K$ is compact and $H$ is the Lie algebra of $K$ seen as a vector space appeared in \cite{AF}.

\begin{lemma}\label{lem:isometry}
Let $K$ be a pseudo-compact Lie group and $H$ be another Lie group such that there is a homomorphism $\rho: K \longrightarrow \mathrm{Aut}(H)$ with $\rho(K)$ pre-compact. If $G$ is the semidirect product $K\ltimes_\rho H$, then there is a left-invariant Riemannian metric on $G$ which is also left-invariant for the direct product $K\times H$.
\end{lemma}

\begin{proof} Since $K$ is pseudo-compact, there exists a bi-invariant Riemannian metric on $K$, but, moreover, since $\rho(K)$ is pre-compact then there exists an $\Ad(K)$-invariant  positive definite inner product on the Lie algebra $\g$ of $G$. Now take $g_R$ to be the Riemannian metric on $G$ induced by left translations. 

The isometry group of $(G,g_R)$ contains not only $G$ (where each $g\in G$ is identified with its left translation) but $K\times G$ where $K$ is acting on $G$ by multiplication from the right. More precisely, the map
\begin{equation*}
(K \times G) \times G \longrightarrow G \quad \text{ such that }\quad   (k, x, y)  \longmapsto xyk^{-1}
\end{equation*}    
is an action of $K\times G$ on $G$ which preserves the Riemannian metric $g_R$. Consider the restriction of this action to $K\times H$. Concretely, we have
  \begin{equation*}
(K \times H) \times G \longrightarrow G \quad \text{ such that }\quad   (k, h, y=(a,b))  \longmapsto hyk^{-1} =(ak^{-1}, hb). 
\end{equation*}    
Clearly this is an action of $K\times H$ on $G$ which is transitive and free thus rendering $G$ into the principal homogeneous space $(K\times H)/\{1_G\}) \cong K\times H$. Furthermore, the map 
 \begin{equation*}
K \times H  \longrightarrow G \quad \text{ such that }\quad   (k, h)  \longmapsto (k^{-1}, h) 
\end{equation*}
is an isometry for the pullback metric of $g_R$ to $K\times H$. Since 
 \begin{equation*}
K  \longrightarrow K \quad \text{ such that }\quad   k  \longmapsto k^{-1} 
\end{equation*}
is also an isometry for the bi-invariant metric on $K$ then the identity map between $K\times H$ and $G$ is also an isometry and the pullback metric on $K\times H$ is thus left-invariant.
\end{proof}

\begin{proposition}\label{prop:bi-Lipschitz}
Let $K$ and $H$ be two Lie groups such that $K$ and $H$ satisfy the conditions of Lem. \ref{lem:isometry}. Then any pair of left-invariant Riemannian metrics on the semidirect product $K\ltimes_\rho H$ and on the direct product $K\times H$ are bi-Lipschitz equivalent. 
\end{proposition}

\begin{proof}
This is an immediate consequence of Lems. \ref{lem:isometry} and \ref{lem: left uniform}.
\end{proof}

\begin{proposition}
Let $G$ be the semidirect product $K \ltimes_\rho V$, where $K$ is a pseudo-compact Lie group and $V$ is a vector space, and let $\rho: K \longrightarrow \mathrm{GL}(V)$ be a representation with $\rho(K)$ pre-compact inside $\mathrm{GL}(V)$.  Then $G$ has linear growth.
\end{proposition}

\begin{proof}
We start by noting that any $(k,v) \in G$ can be written as $(1,v)(k,0)$ and therefore $\Ad_{(k,v)} = \Ad_{(1,v)}\Ad_{(k,0)}$. Thus, for $(u,w)\in \mk\ltimes_{\rho_\ast} V $, 
\begin{equation*}
\Vert \Ad_{(k,v)}(u,w) \Vert \leq \Vert \Ad_{(k,v)} \Vert \Vert (u,w) \Vert \leq \Vert \Ad_{(1,v)} \Vert \Vert \Ad_{(k,0)} \Vert \Vert (u,w) \Vert. 
\end{equation*}
Since $\mathrm{Ad}(K)$ and $\rho(K)$ are pre-compact then $k\longmapsto \Vert \Ad_{(k,0)} = (\mathrm{Ad}_k, \rho(k))\Vert$ is bounded and, therefore, it is sufficient to consider the action of $V$ via the adjoint representation. 
We have that 
\begin{equation*}
\Ad_{(1,v)}(u,w) = (u, w) + (0, -\rho_\ast(u)(v))
\end{equation*}
which can be rewritten as
\begin{equation*}
\Ad_{(1,v)}(u,w) = (\Id + \ad_{(0,v)})(u, w)
\end{equation*}
and thus
\begin{equation*}
\Vert \Ad_{(1,v)}(u,w) \Vert \leq (\Vert  \Id \Vert + \Vert \mathrm{ad}_{(0,v)} \Vert )\Vert (u,w) \Vert.   
\end{equation*}
From the bilinearity of $\rho_\ast$ we get that  
$$\Vert \mathrm{ad}_{(0,v)} (u,w)\Vert = \Vert -\rho_\ast(u)v\Vert \leq \Vert \rho_\ast \Vert \Vert u \Vert \Vert v \Vert.$$  Then taking the operator norm $\Vert \ad_{(0,v)} \Vert \leq b \Vert v \Vert$, with $b \leq \Vert \rho_\ast \Vert$. For the standard metric on $\R^n$ we have $\Vert v \Vert = r^V(v)$. Thus 
\begin{equation*}
\Vert \Ad_{(1,v)}(u,w) \Vert \leq (a + b\, r^V(v))\Vert (u,w) \Vert   
\end{equation*}
with $a = \Vert \Id \Vert$. Taking $r^{K\times V}$  the distance with respect to a (left-invariant) product metric, we obtain
\begin{equation*}
\Vert \Ad_{(1,v)}(u,w) \Vert \leq (a + b\, r^{K\times V}(k,v))\Vert (u,w) \Vert.   
\end{equation*}
 We can now use Prop. \ref{prop:bi-Lipschitz} and this yields the result.
\end{proof}

Combining the results of this section, we have now completed the list in  Th. \ref{main} and therefore its proof.  We finish this section with a class of groups which not only do not have linear growth but have, in fact, exponential growth. 

\subsection{Existence of idempotents}

Recall, as observed in \eqref{eq:EA-2nil}, that given a Lie group $G$ equipped with a left-invariant semi-Riemannian metric $g$, the geodesics of $(G,g)$ are in one-to-one correspondence with the solutions of the Euler-Arnold equation
\begin{equation}\label{eq:EA}
\dot{x} = \ad_x^\dagger x
\end{equation}
on the Lie algebra $\g$, \cite{Arn}, where now $\dagger$ denotes the adjoint with respect to the pseudo-scalar product $g_1$.  It is well-known that the existence of an idempotent, that is, an $x_0 \neq 0$ such that $\ad^\dagger_{x_0}x_0 = x_0$, gives an incomplete solution of \eqref{eq:EA},  cf. \cite{BM}. Observe also that such a geodesic (for instance, $\gamma(t) = \exp(-\ln(1-t) x_0)$) gives a divergent curve of finite length for the Clairaut metric  associated to $g$.  Using Th. \ref{t_lg}, such a group $G$ cannot, therefore, have linear growth. In fact, this can be proven directly, and what is more, the existence of an idempotent implies exponential growth, as we shall see in the following.

\begin{proposition}\label{prop:idemp-exp-growth}
Let $G$ be a Lie group that can be equipped with a left-invariant semi-Riemannian metric $g$ for which the Euler-Arnold equation has an idempotent.  Then, there exists a curve such that $G$ has exponential growth along that curve. 
\end{proposition}

\begin{proof}
We start by proving that if $x_0\in \g$ is an idempotent then $\ad_{x_0}$ has an eigenvector $y_0$ with eigenvalue 1. From the definition of idempotent we have that, for any $y\in \g$, $g_1(x_0, y -\ad_{x_0} y) = 0$. Thus the map $y \longmapsto y-\ad_{x_0} y $ has non-trivial kernel and there exists a $y_0$ such that $\ad_{x_0}y_0 = y_0$.  
Take any left-invariant Riemannian metric $g_R$ on $G$ and consider  the curve $[0,\infty) \ni t \mapsto p(t)= \exp\left( t x_0 \right)$.  Assuming, without loss of generality, that $\Vert x_0 \Vert =1$, hence, as seen before,  $d_R(1, p(t) ) \leq t$.   Moreover, 
\begin{equation*}
\Ad_{ p(t) }(y_0) = \exp(\ad_{t x_0})y_0 = \left(\Id +t \ad_{x_0} + \frac{t^2}{2}\ad_{x_0}^2 + \frac{t^3}{3!} \ad_{x_0}^3+\cdots \right) y_0  = \e^t y_0.
\end{equation*} 
 We then have exponential growth along $p(t)$, as  $\Vert \Ad_{p(t)}(y_0) \Vert  \geq \e^{d_R(1,p(t))}\Vert y_0 \Vert.$
\end{proof}

\section{Example: the affine group of the real line}\label{s6}

 The affine group does not satisfy the linear growth condition and admits both complete and incomplete left-invariant metrics. We will study this group  with two aims. 
First, it provides an interesting illustration of the properties of the Clairaut metrics, their possible growth and completeness (these possibilities were stressed in Remark \ref{r_BEWARE}).

Second, but no less important,  the technique permits to determine exactly all its complete and incomplete left-invariant metrics.
More precisely, our aim will be to divide these metrics into three classes, depending on the behaviour of their Clairaut metrics. One of them is the class of the positive definite (and negative definite) ones; such metrics are complete and so are their Clairaut metrics.   
The other two classes include all the indefinite (necessarily Lorentzian) metrics and they have incomplete left-invariant Clairaut metrics.  The existence of one incomplete representative in each class implies the incompleteness of all these metrics.

\subsection{Framework}
Consider $\mathrm{Aff}(\R)$ the group of affine transformations of the real line: $f(x) = a x + b$, with $a\neq 0$. Take the connected component of the identity, which we will denote by $\mathrm{Aff}^+(\R)$. It can be seen as semidirect product $\R^+ \ltimes \R$ 
and as a matrix group can be concretely written as 
\begin{equation*}
\mathrm{Aff}^+(\R) = \left\{ \left(  \begin{array}{cc}
x & y \\ 0 & 1 
\end{array} \right) \, : \, x>0, y\in \R \right\}
\end{equation*}
with global coordinates $(x,y)$. 
Its Lie algebra, denoted by $\mathfrak{aff}(\R)$ is thus
\begin{equation*}
\mathfrak{aff}(\R) = \left\{ \left(  \begin{array}{cc}
u & v \\ 0 & 0 
\end{array} \right) \, : \, u,v\in \R \right\}
\end{equation*}
with the Lie bracket given by the usual commutator. 
Choosing the matrices 
\begin{equation*}
e_1 = \left( \begin{array}{cc}
1 & 0 \\ 0 & 0
\end{array} \right) \quad \text{and} \quad e_2 = \left( \begin{array}{cc}
0 & 1 \\ 0 & 0
\end{array} \right) 
\end{equation*}
as generators of $\mathfrak{aff}(\R)$, they satisfy the relation $[e_1, e_2] = e_2$. An easy computation shows that $X_1 = x \partial_x$ and $X_2 = x \partial_y$ are the left-invariant vectors fields induced by $e_1$ and $e_2$, respectively.

Let $g$ be any left-invariant metric on $\mathrm{Aff}^+(\R)$. Using the fact that $g(X_i, X_j)$ are constant functions, we can compute that $g$ is given by
\begin{equation}\label{e_gen_aff_inv_g}
g= \frac{1}{x^2} (c_1 \, dx^2 + c_2 \, ( dx dy +dy dx) +  c_3 \, dy^2 )
\end{equation}
where $c_1, c_2, c_3$ are constants verifying $c_1 c_3 - c_2^2 \neq 0$ (imposing the non-degeneracy of the metric).

\subsection{Two left-invariant metrics  $g^{(\epsilon)}$ and their completeness. 
}  
Let us consider the following choices for the parameters in \eqref{e_gen_aff_inv_g},
\begin{equation*}
c_1=1, c_2=0, c_3=\epsilon ,
\end{equation*}
with  $\epsilon=\pm 1$. The corresponding left-invariant metrics will then be:
$$
g^{(\epsilon)}=\frac{1}{x^2}(dx^2+\epsilon dy^2).
$$
Notice that $g^{(1)}$ is complete since it is a  left-invariant Riemannian metric (what is more,   $g^{(1)}$ is a hyperbolic metric). However, the Lorentzian metric  $g^{(-1)}$ is incomplete; indeed,  a simple computation shows that the curve $\gamma(t) = \left( \frac{1}{1-t} , \frac{1}{1-t} \right)$ is an incomplete geodesic for $g^{(-1)}$. This curve comes from an idempotent of the Euler-Arnold vector field, but can also be checked directly by computing the Christoffel symbols and substituting in the geodesics equations.

\subsection{Clairaut 
$h^{(\epsilon)}$: equal non-linear growth}
We now compute the Clairaut metrics $h^{(\epsilon)}$ for $g^{(\epsilon)}.$ The right-invariant vector fields induced by $e_1$ and $e_2$ are given, respectively, by $Y_1 = x \partial_x + y \partial_y$  and $Y_2 =  \partial_y$. Thus the Clairaut first integrals are given by 
\begin{equation*}
\omega^1 = \frac{1}{x^2}(x dx + \epsilon dy)  \quad \text{and} \quad \omega^2 = \frac{\epsilon}{x^2} 
\end{equation*}
and the Clairaut metric $h^{(\epsilon)} = (\omega^1)^2 + (\omega^2)^2$ is thus

\begin{equation}\label{ehe}
h^{(\epsilon)}= \frac{1}{x^4}\left(x^2 dx^2+(1+y^2)dy^2+ \epsilon xy (dx dy+dy dx)\right). 
\end{equation}
Notice that the matrix of $h^{(\epsilon)}$ in these coordinates $(x,y)$ is:
\begin{equation}\label{e_matrix}
\frac{1}{x^4}
\begin{pmatrix}
x^2 & \epsilon xy \\
\epsilon xy & 1+y^2
\end{pmatrix}.
\end{equation}

Computing the determinant and eigenvalues evl$_\pm$  of this matrix:
\begin{equation}
\label{evl}
\hbox{det}=\frac{\epsilon^2}{x^6}, \qquad \hbox{evl}_\pm=
\frac{1}{2x^4}\left(x^2+\epsilon^2(1+y^2)\pm \sqrt{(x^2+\epsilon^2(1+y^2))^2-4\epsilon^2 x^2} \right)
\end{equation}
(clearly, $0<$ evl$_-<$ evl$_+$). 

The eigenvalues measure the growth of $h^{(\epsilon)}$ with respect to  the usual metric $dx^2+dy^2$. 
Notice that the left-invariant metric $g^{(1)}$ is conformal to the usual one and, thus,  up to a conformal factor, these eigenvalues also measure the growth of $h^{(\epsilon)}$ with respect to $g^{(1)}$  (see  details in \S \ref{s6.5}). 

It is interesting to observe that the eigenvalues are independent of sign$(\epsilon)$. Thus, the growth of $h^{(\epsilon)}$ with respect to $g^{(1)}$ is independent of sign$(\epsilon)$.  As $g^{(-1)}$ is incomplete, this growth cannot be linear  for $g^{(-1)}$ and, thus, neither   for $g^{(1)}$. This will be proved directly in \S\ref{s6.5} (indeed,  exponential growth is proven, which is compatible with the existence of idempotents, cf. Prop. \ref{prop:idemp-exp-growth})  and  exemplifies a general property established in Remark \ref{r_BEWARE}(1). 

\subsection{Heuristics on  eigendirections} \label{s6.4} As $g^{(-1)}$ is incomplete, we know  that so is $h^{(-1)}$. We have also seen that 
the growth of the Riemannian metrics $h^{(\epsilon)}$ is equal for $\epsilon=\pm 1$. However, this does not imply that $h^{(1)}$ must be complete. As pointed out in Remark \ref{r_BEWARE}(2) the direction of the eigenvectors for evl$_-$ is crucial at this point. Indeed, in our example the eigenvectors do depend on sign$(\epsilon)$  and $h^{(1)}$ will be complete.

Even though working directly with the eigenvectors is complicated, here, we can understand the qualitative behaviour by rewriting \eqref{e_matrix} as
$$
\frac{1}{x^2}
\begin{pmatrix}
1 & \epsilon \frac{y}{x} \\
\epsilon \frac{y}{x} & \epsilon^2
\frac{1+y^2}{x^2}
\end{pmatrix}
= \frac{1}{x^2}
\begin{pmatrix}
0 & 0 \\
0 & \frac{\epsilon^2}{x^2} 
\end{pmatrix}
+
\frac{1}{x^2}
\begin{pmatrix}
1 \\ \epsilon \frac{y}{x} 
\end{pmatrix}
\begin{pmatrix}
1 & \epsilon \frac{y}{x} 
\end{pmatrix}.
$$
The last term is a singular matrix with pointwise kernel spanned  by the vector field:
$$
V:=(-\epsilon y, x)\equiv -\epsilon y\partial_x+ x\partial_y . 
$$
It may seem that, following the integral curves of $V$, the metric $h^{(\epsilon)}$ will be small and, then, we {\it might} find an incomplete curve. As we will see, the behaviour of these integral curves changes  with sign$(\epsilon)$. \   The integral curves satisfy:
\begin{equation}\label{e_intcurvV}
\left. \begin{array}{c}
\dot x= -\epsilon y \\ 
\dot y =x \end{array} \right\}
\qquad \hbox{thus}, \quad
\left. 
\begin{array}{c}
\ddot x= -\epsilon x \\ 
\ddot y = -\epsilon y\end{array}\right\}.
\end{equation}
From the last equation,
the expected behaviour is then:
\bit
\item For $\epsilon=-1 , V=y\partial_x+x\partial_y $. There exist diverging integral curves with finite length, thus  proving  incompleteness of $h^{(-1)}$. 
\item For $\epsilon=1, V=-y\partial_x+x\partial_y (\equiv r\partial_\theta) $. Its  integral curves are trigonometric and  cannot yield incompleteness. 
\eit

\subsection{Metrics $h^{(\epsilon)}$: explicit (in-)completeness and exponential growth}\label{s6.5}  Let us rigorously prove the claimed behaviours for $h^{(\epsilon)}$.

\subsubsection*{Case $\epsilon=-1$ (incompleteness)}
Consider the following curve solving \eqref{e_intcurvV}:
$$
\gamma(t)=(x(t),y(t))=(\cosh t, \sinh t) \quad [\hbox{so} \; \dot \gamma(t)=(\dot x(t),\dot y(t))=(\sinh t, \cosh t)]  \qquad \forall t\geq 0 .
$$
It is clearly diverging, and showing that it has finite length suffices. Using \eqref{ehe}:
$$\begin{array}{rl}
h^{(\epsilon=-1)}(\dot \gamma(t),\dot \gamma(t))= & \frac{1}{\cosh^4 t}
\left(
\cosh^2 t \sinh^2 t+(1+\sinh^2t)\cosh^2t
-2 \cosh^2 t \sinh^2 t 
\right)\\
=&  \frac{1}{\cosh^4 (t)}\cosh^2 (t)=
\frac{1}{\cosh^2 (t)}
\end{array}
$$
and the result follows directly, since 
\begin{equation*}
\mathrm{length}(\gamma) = \int_{0}^{+\infty} \frac{1}{\cosh(t)}\, dt \leq \int_{0}^{+\infty} 2 \mathrm{e}^{-t}\, dt = 2.
\end{equation*}

\subsubsection*{Case $\epsilon=1$ (completeness)}
In order to prove completeness, let us consider any diverging  curve $\gamma(t)=(x(t),y(t)), t\in [0,b), b\leq \infty$
and check that its length is infinite.

We claim that to consider only curves $\gamma(t)=(x(t),y(t))$ with  $r^2(t)=x(t)^2+y(t)^2$ unbounded is sufficient. Indeed,  the minimum eigenvalue of $h^{(1)}$ in \eqref{evl} is  such that 
\begin{equation}\label{e_evl}
\begin{array}{lcl} \hbox{evl}_- & = & 
\frac{1}{2x^4}\left((1+x^2+y^2)- \sqrt{(1+x^2+y^2)^2-4 x^2}\right) \\
& = & \dfrac{4 x^2}{2x^4\left(1+x^2+y^2+ \sqrt{(1+x^2+y^2)^2-4 x^2}\right)} \\ & \geq &   
\dfrac{1}{x^2(1+x^2+y^2)}.  \end{array}
\end{equation}
  So, if $x(t), |y(t)|<C$, 
$$
h^{(1)}\geq \frac{dx^2+dy^2}{x^2(1+x^2+y^2)}\geq  \frac{dx^2+dy^2}{x^2}\frac{1}{(1+2C^2)}
$$
which is a (complete) hyperbolic  metric.

Now, using again \eqref{ehe} and putting $r^2=x^2+y^2, (r^2)'=2x\dot x+2y\dot y$:
$$\begin{array}{rl}
h^{(\epsilon=1)}(\dot \gamma(t),\dot \gamma(t))= & \frac{1}{x^4} (x^2 \dot x^2+ (1+y^2)\dot y^2+2x\dot xy \dot y)
 = \frac{1}{x^4} \left((x\dot x+y\dot y)^2 + \dot y^2\right)
 \\ \geq  &
\frac{1}{x^4} (x\dot x+y\dot y)^2 
= \frac{1}{x^4} (\frac{1}{2}(r^2)')^2 \\ \geq & \frac{1}{r^4} (\frac{1}{2}(r^2)')^2 = (\frac{1}{2 r^2}(r^2)')^2  
\\ = &  \left (\frac{1}{2} (\ln (r^2))'\right)^2. 
\end{array}
$$
Thus, taking $t_n \nearrow b$ such that 
 $\{\gamma(t_n)\}_n$ is an unbounded sequence:
$$\begin{array}{rl}
\hbox{length}(\gamma) \geq & \lim_{t_n\rightarrow b} \frac{1}{2}\int_0^{t_n}   (\ln(r^2))'(t) dt \\
 = & \frac{1}{2}\left( \lim_{t_n\rightarrow b}\ln (r^2(t_n))-\ln (r^2(0))\right) = \lim_{t_n\rightarrow b}\ln (r(t_n))-\ln (r(0))=\infty,
\end{array}$$
i.e., it goes to infinity albeit slowly.

 \subsubsection*{Exponential growth}  Let us check the claimed growth on the $x$-semiaxis. According to \eqref{e:lambda-},
$$ \Vert \Ad_{(x,0)}\Vert= \frac{1}{\lambda_-((1/x),0)}$$
where, from \eqref{e_sharp1}, $\lambda^2_-((1/x),0)$ is the  minimum eigenvalue of $h^{(\epsilon)}_{(x,0)}$ with respect to the Riemannian metric $\tilde{g}^{(-1)}$, which in our case is equal to 
$g^{(1)} (=\tilde{g}^{(1)})$, i.e., it is the minimum eigenvalue of $\phi_{(x,0)}$ defined on Aff$^+(\R)$ from
$$
h^{(\epsilon)}_{(x,y)}(u,v)=g^{(1)}_{(x,y)}(u,\phi_{(x,y)}(v)), \qquad \qquad \qquad \forall (x,y)\in \R^+\times \R, \quad \forall u,v\in \R^2. 
$$
The eigenvalue evl$_-$ was computed with respect to $dx^2+dy^2=x^2 g^{(1)}$, thus $\lambda^2_-((1/x),0)= x^2 
\hbox{evl}_{-(x,0)}$   and using the second equality in \eqref{e_evl}
$$
\lambda^2_-((1/x),0)=
\frac{2}{1+x^2+\sqrt{(1-x^2)^2}}, \qquad \hbox{that is}, \qquad
\Vert \Ad_{(x,0)}\Vert=\frac{1}{\sqrt{2}}\sqrt{1+x^2+|1-x^2|}.
$$
To check that this behaviour is not linear for $g^{(1)}$, notice that $\gamma(t)=(e^t,0)$ is a unit (minimizing) geodesic for this metric and $r(\gamma(t))=d_{g^{(1)}}((1,0), \gamma(t)) = |t|$, thus,
$$\lim_{t\rightarrow \infty}\frac{\Vert \Ad_{\gamma(t)}\Vert}{e^{r(t)}}=\lim_{t\rightarrow \infty}\frac{\sqrt{1+e^{2t}+|1-e^{2t}|}}{\sqrt{2} \, e^t}=1.
$$

\subsection{A third  left-invariant metric $g^{(0)}$} 
Consider also the following choices for the parameters
\begin{equation*}
c_1=0, c_2=1, c_3=0.
\end{equation*}
The corresponding left-invariant metric will then be:
$$
g^{(0)}=\frac{1}{x^2}(dxdy+dydx).
$$
Another computation shows that the curve $\gamma(t) = \left( \frac{1}{1-t} , 0 \right)$ is an incomplete geodesic for  $g^{(0)}$ and, thus, its Clairaut metric $h^{(0)}$ 
must then be incomplete. This metric  is
\begin{equation*}
h^{(0)}= \frac{1}{x^4}\left((1+y^2) dx^2+ x^2 dy^2+  xy (dx dy+dy dx)\right) 
\end{equation*}
and we can check directly its incompleteness, since the curve $\gamma: [1, +\infty[ \longrightarrow \R$ such that $\gamma(t) = (t,0)$  is diverging 
and has finite $h^{(0)}$-length. 
Notice that this curve has infinite length for the Clairaut metric $h^{(-1)}$ in \eqref{ehe} and, thus,  
{\it $h^{(0)}$ and $h^{(-1)}$ are not bi-Lipschitz related}.

\subsection{The three classes of left-invariant metrics and their completeness}
We will now show how the analysis of these three choices of parameters allows us to further conclude that all the Clairaut metrics associated to Riemannian left-invariant metrics are complete and all the Clairaut metrics associated to left-invariant semi-Riemannian metrics are incomplete. 

Following \S \ref{s_Abstractformalism}, next we consider the action of the automorphism group of $\g = \mathfrak{aff}(\R)$, $\mathrm{Aut}(\g)$, on $\Sym(\g)$.
If $\varphi: \g \longrightarrow \g$ is an automorphism of $\g$, it satisfies $[\varphi(e_1), \varphi(e_2)] = \varphi(e_2)$ and with respect to the basis $\{e_1, e_2\}$ can be represented by the matrix 
\begin{equation*}
M= \begin{pmatrix}
1 & 0 \\
\alpha & \beta 
\end{pmatrix}, \qquad \beta \neq 0.
\end{equation*}
Then $\varphi^{-1}\in\Aut(\g)$ acts on $\mathrm{Sym}(\g)$ by $M^T B M$, for a bilinear form $B$. We will study the orbits of the action of $\mathrm{Aut}(\g)$ for non-degenerate forms on $\mathrm{Sym}(\g)$. As we will see the Euclidean scalar products are all in the same orbit and the indefinite scalar products are in two orbits, which correspond to the cases where $e_2$ is isotropic or non-isotropic. 
\begin{itemize}
\item Orbit of $B= \begin{pmatrix}
1 & 0 \\ 0 & 1
\end{pmatrix}$:

\smallskip 

$M^T B M = \begin{pmatrix}
1 & 0 \\
\alpha & \beta 
\end{pmatrix}^T \begin{pmatrix}
1 & 0 \\ 0 & 1
\end{pmatrix}  \begin{pmatrix}
1 & 0 \\ \alpha & \beta
\end{pmatrix}  =  \begin{pmatrix}
1+ \alpha^2 & \alpha\beta \\ \alpha\beta & \beta^2
\end{pmatrix}$ 

\smallskip
\noindent and thus, \emph{up to (positive or negative)  scaling}, this orbit  contains Euclidean 
 (or negative definite)  scalar products and, moreover, all of them. The last assertion can be seen by directly solving the equation $M^T B M = \lambda P$, where $P$ is the matrix of a general Euclidean scalar product, and taking $\lambda\in\R\backslash\{0\}$.  

\item Orbit of $B= \begin{pmatrix}
1 & 0 \\ 0 & -1
\end{pmatrix}$:

\smallskip 

$M^T B M = \begin{pmatrix}
1 & 0 \\
\alpha & \beta 
\end{pmatrix}^T \begin{pmatrix}
1 & 0 \\ 0 & -1
\end{pmatrix}  \begin{pmatrix}
1 & 0 \\ \alpha & \beta
\end{pmatrix}  =  \begin{pmatrix}
1- \alpha^2 & -\alpha\beta \\ -\alpha\beta & -\beta^2
\end{pmatrix}$ 

\smallskip
\noindent and,  reasoning as above,  \emph{up to scaling}, this orbit corresponds to all  Lorentzian  scalar products such that $\langle e_2, e_2 \rangle \neq 0$.

\item Orbit of $B= \begin{pmatrix}
0 & 1 \\ 1 & 0
\end{pmatrix}$:

\smallskip 

$M^T B M = \begin{pmatrix}
1 & 0 \\
\alpha & \beta 
\end{pmatrix}^T \begin{pmatrix}
0 & 1 \\ 1 & 0
\end{pmatrix}  \begin{pmatrix}
1 & 0 \\ \alpha & \beta
\end{pmatrix}  =  \begin{pmatrix}
2\alpha & \beta \\ \beta & 0
\end{pmatrix}$ 

\smallskip
\noindent and thus this orbit corresponds to all   Lorentzian  scalar products such that $\langle e_2, e_2 \rangle = 0$.
\end{itemize}

 As we have found one incomplete geodesic in each class of Lorentzian orbits of Aff$^+(\R)$, Prop. \ref{p_aut_orbits}(2)  proves that all the left-invariant Lorentzian metrics on Aff$^+(\R)$ are incomplete and, therefore Cor. \ref{Cor:affine-group}. Furthermore, Prop. \ref{p_aut_orbits}(3) shows the existence of only three bi-Lipschitz classes of Clairaut metrics.

\section*{Acknowledgements}  

\begingroup
\sloppy

 The research of the first and second named authors was financed by Portuguese Funds through FCT (Fundação para a Ciência e a Tecnologia, I.P.) within the Projects UIDB/00013/2020 and UIDP/00013/2020 with the references  DOI: 10.54499/UIDB/00013/2020 (https://doi.org/10.54499/UIDB/00013/2020) and   DOI: 10.54499/UIDP/00013/2020 (https://doi.org/10.54499/UIDP/00013/2020). The third named author was partially supported by the grants, 
PID2020-116126GB-I00 (MCIN/AEI/10.13039/501100011033),
and the framework IMAG/Mar\'{\i}a
de Maeztu, 
CEX2020-001105-MCIN/AEI/10.13039/501100011033.

\endgroup


\end{document}